\documentclass[twoside, reqno, 10pt]{amsart}

\usepackage[left=3.8cm, right=3.8cm, top=3cm, bottom=2.5cm]{geometry}

\usepackage{algorithm}
\usepackage{algpseudocode}

\usepackage[colorlinks=true, pdfstartview=FitV, linkcolor=blue,citecolor=blue, urlcolor=blue]{hyperref}

\usepackage{xcolor}

\definecolor{labelkey}{rgb}{0,0,1}
\definecolor{Red}{rgb}{0.7,0,0.1}
\definecolor{Green}{rgb}{0,0.7,0}

\usepackage{amsfonts, amssymb, amsmath, amsthm, mathrsfs, bbm, cjhebrew, gensymb, textcomp, mathtools, dsfont, calligra}

\usepackage[normalem]{ulem}

\usepackage{commath, setspace, subcaption, parcolumns, multirow, multicol, accents, comment, marginnote, verbatim, empheq, enumerate, stackrel, enumitem, float, physics, todonotes}
\usepackage[capitalize,nameinlink,noabbrev]{cleveref}

\usepackage{graphicx, graphics, epsfig, psfrag, tikz, tikz-cd,svg}

\numberwithin{equation}{section}

\usepackage{tikz}

\newtheorem{Thm}{Theorem}[subsection]
\newtheorem{Lem}[Thm]{Lemma}
\newtheorem{Prop}[Thm]{Proposition}
\newtheorem{Cor}[Thm]{Corollary}

\newtheorem{Def}[Thm]{Definition}
\newtheorem{Rmk}[Thm]{Remark}

\newtheorem*{Thm*}{Theorem}

\newcommand{\ZZ}{\mathbb{Z}}

\newcommand{\RR}{\mathbb{R}}

\newcommand{\PP}{\mathbb{P}}

\newcommand{\cU}{\mathcal{U}}

\DeclareMathOperator{\diag}{diag}
\DeclareMathOperator{\Id}{Id}

\newcommand{\no}[2]{\lVert#2\rVert_{#1}}
\newcommand{\Abs}[2]{\lvert#2\rvert_{#1}}

\newcommand{\goesto}{\rightarrow}
\newcommand{\smod}{\setminus}

\newcommand{\al}{\alpha}
\newcommand{\be}{\beta}
\newcommand{\de}{\delta}
\newcommand{\De}{\Delta}

\newcommand{\eps}{\epsilon}

\newcommand{\si}{\sigma}

\newcommand{\tht}{\theta}

\newcommand{\Om}{\Omega}

\newcommand{\bdy}{\partial}

\newcommand{\lp}{\left(}
\newcommand{\rp}{\right)}
\newcommand{\lpp}{(\!(}
\newcommand{\rpp}{)\!)}

\newcommand{\rt}{\tilde{r}}
\newcommand{\tv}{\tilde{v}}
\newcommand{\tf}{\tilde{f}}

\newcommand{\tq}{\tilde{q}}

\newcommand{\Gr}{\mathfrak{G}}

\newcommand{\cD}{\mathcal{D}}

\newcommand{\cUt}{\widetilde{\mathcal{U}}}

\newcommand{\cM}{\mathcal{M}}

 \title[Determining Modes, State Reconstruction, and Intertwinement: Part 1]{Determining Modes, State Reconstruction, and Intertwinement: A Synchronization Framework}

 \author{Elizabeth Carlson, Aseel Farhat, Vincent R. Martinez$^\dagger$, Collin Victor}

\date{December 9, 2025\\
\indent $^\dagger$Corresponding author}

\begin{document}

\begin{abstract}
This article studies the interrelation between the determining modes property in the two-dimensional (2D) Navier-Stokes equations (NSE) of incompressible fluids and the reconstruction property of two filtering algorithms for continuous data assimilation applied to the 2D NSE. These two properties are realized as manifestations of a more general phenomenon of \textit{self-synchronous intertwinement}. It is shown that this concept is a logically stronger form of asymptotic enslavement, as characterized by the existence of finitely many determining modes in the 2D NSE. In particular, this stronger form is shown to imply convergence of the direct-replacement filter and the nudging filter from continuous data assimilation (CDA), and then subsequently invoked to show that convergence in these filters implies that the 2D NSE possesses finitely many determining modes. The main achievement of this article is to therefore to develop a new conceptual framework, that of self-synchronous intertwinement, through which the precise inter-relationship between the determining modes property and synchronization phenomenon in these CDA filters is rigorously established and made decisively clear. The theoretical results are then complemented by numerical experiments that confirm the conclusions of the theorems. 
\end{abstract}

\maketitle

\vspace{1em}

{\noindent \small {\it {\bf Keywords: intertwinement, synchronization, determining modes, continuous data assimilation, nudging filter, direct-replacement filter, Navier-Stokes equations}
  }  
  }


\section{Introduction}\label{sect:intro}
In a seminal paper by C. Foias and G. Prodi \cite{FoiasProdi1967}, it was shown that the two-dimensional (2D) externally forced, Navier-Stokes equations (NSE) for incompressible fluids, given by the system
    \begin{align}\label{eq:nse}
        \bdy_tu+u\cdotp\nabla u=-\nabla p+\nu\De u+f,\quad \nabla\cdotp u=0,
    \end{align}
asymptotically possesses a large, but finite number of degrees of freedom, a property that is expected to hold true for turbulent flows on the basis of physical principles. In particular, they introduce the notion of \textit{determining modes} and subsequently show that the 2D NSE possesess finitely many such modes. This notion characterizes the property that knowledge of the asymptotic behavior of a distinguished set of constitutive modes of the solution suffice to describe the asymptotic behavior of all 
of its modes. In this way, one therefore captures the finite-dimensionality of the dynamics of the system. It is remarkable that this fact was established at the advent of the study of chaotic dynamical systems, just a few years after E. Lorenz introduced his three-mode truncation of the Boussinesq approximation of the full Navier-Stokes equations for weather prediction \cite{Lo1963}. It had also provided strong evidence for the finite dimensionality of the global attractor of the 2D incompressible NSE, which was eventually resolved in the two decades following \cite{FoiasProdi1967} in \cite{FoiasTemam1979, Ladyzhenskaya1985, ConstantinFoiasTemam1988}.

This result has since been realized as a cornerstone in the study of the dynamics of the 2D NSE and many other dissipative partial differential equations (PDEs), deterministic and stochastic. It has either motivated or found intimate connection to the study of finite dimensionality of global attractors \cite{FoiasTemam1979, Ladyzhenskaya1985, ConstantinFoiasTemam1988}, estimates on dissipative length scales of turbulent flows \cite{FoiasManleyTemamTreve1983, FoiasTemam1985, ConstantinFoiasManleyTemam1985}, the study of dimension-reduction, approximate inertial manifolds, and downscaling in dissipative systems \cite{FoiasTiti1991, OlsonTiti2003, AzouaniOlsonTiti2014, KalantarovKostiankoZelik2023}, and the existence of determining forms \cite{FoiasJollyKravchenkoTiti2012, FoiasJollyKravchenkoTiti2014, JollySadigovTiti2015, JollySadigovTiti2017, FoiasJollyLithioTiti2017, JollyMartinezSadigovTiti2018}, a notion weaker than that of an inertial form, in which the dynamics of the underlying PDE system is reduced to the study a bona fide ODE system, albeit an infinite-dimensional one. Determining modes have also found a wide range of applications in the domains of data assimilation \cite{BlomkerLawStuartZygalakis2013}, numerical approximation of PDEs \cite{MondainiTiti2018, IbdahMondainiTiti2019}, parameter estimation \cite{CarlsonHudsonLarios2020, CarlsonHudsonLariosMartinezNgWhitehead2021, PachevWhiteheadMcQuarrie2022, Martinez2022, BiswasHudson2023, Martinez2024, FarhatLariosMartinezWhitehead2024, AlbanezBenvenutti2024}, and perhaps most notably, the problem of unique ergodicity for stochastically forced systems \cite{EMattinglySinai2001, KuksinShirikyan2012, Debussche2013}, where it is often referred to as the \textit{Foias-Prodi property}. In the context of the 2D NSE, the concept of determining modes has also been extended to accommodate the concepts of determining nodes, volume elements, and, more generally, determining functionals \cite{FoiasTemam1984, CockburnJonesTiti1997, JonesTiti1992a, JonesTiti1992b, ChueshovDuanSchmalfuss2003}, which have all enjoyed a richness in application 
\cite{FlandoliLanga1999, Langa2003, HaydenOlsonTiti2011, FarhatJollyTiti2015, AlbanezLopesTiti2016,  FoiasMondainiTiti2016, JollyMartinezTiti2017, AltafTitiGebraelKnioZhaoMcCabe2017, FarhatJohnstonJollyTiti2018, LeoniMazzinoBiferale18, DiLeoniClarkMazzinoBiferale2018, OljacaBrockerKuna2018, BiswasFoiasMondainiTiti2018, ZerfasRebholzSchneierIliescu2019, IbdahMondainiTiti2019, LariosRebholzZerfas2019, CelikOlsonTiti2019, Garcia-ArchillaNovo2020, BiswasBradshawJolly2021, CelikOlson2023, CaoGiorginiJollyPakzad2022, HammoudTitiHoteitKnio2022, JollyPakzad2023, Garcia-ArchillaLiNovoRebholz2024, KoronakiEvangelouMartin-LinaresTitiKevrekidis2024}.

In some of these applications, the existence of finitely many determining modes is not invoked explicitly, but rather appears implicitly through the analysis. Although it is by now well-known how the Foias-Prodi mechanism appears in these applications, as of yet, there has been no result that {demonstrates a form of equivalence between the} Foias-Prodi property of the Navier-Stokes equations to the particular application in which its mechanism appears. Indeed, in the context of either data assimilation or unique ergodicity, such a relation {exists in a moral sense through observing} how closely the analysis {in the one context} matches {the analysis} carried out in \cite{FoiasProdi1967}, or in anticipating how solutions should behave as a consequence of the property, thereby informing the approach that is developed for the application itself. For example, in proving that the continuous data assimilation algorithms of E. Olson and E. Titi \cite{OlsonTiti2003} or A. Azouani, E. Olson, and E. Titi \cite{AzouaniOlsonTiti2014} are capable of asymptotically recovering the reference state variable from partial observations of the system, i.e., achieving \textit{synchronization} of the data assimilated approximating state variable with the reference state variable, one observes strong similarities to the argument of C. Foias and G. Prodi in \cite{FoiasProdi1967}. {One notable contribution in establishing a more rigorous link between the existence of the Foias-Prodi mechanism in a given system and the reconstruction property of certain continuous data assimilation (CDA) algorithms is found in \cite{BiswasHudsonLariosPei17}; there it is shown that the nudging algorithm can be leveraged to demonstrate that the underlying system possesses the determining modes property. To our knowledge, \cite{BiswasHudsonLariosPei17} is the first to establish a rigorous connection between the determining modes property and the ability of a CDA algorithm to synchronize.}  However, whether or not {the \textit{converse} is true or if such relations hold for other CDA algorithms other than the nudging algorithm} has remained unaddressed. {In this article, we develop a general framework in which we are able not only to formalize, but also to prove that the corresponding converse statement holds for \textit{both} CDA algorithms introduced in \cite{OlsonTiti2003} and \cite{AzouaniOlsonTiti2014}. As a byproduct of our analysis, our framework and results effectively subsume  \cite[Theorem 3.10, 3.12]{BiswasHudsonLariosPei17} corresponding to determining modes in the context of the 2D NSE.}

{On the other hand, in the context of unique ergodicity of degenerately forced stochastic} systems, asymptotic couplings or exact finite-dimensional couplings have been constructed to deduce uniqueness of invariant probability measures for its Markovian dynamics, and even deduce mixing rates. Due to the degeneracy of the noise, the design of these couplings are restricted to enforcing a form of contractivity only on a finite, but potentially large, dimensional subspace. However, for systems possessing the Foias-Prodi property, contractivity on a finite-dimensional subspace subsequently activates a nonlinear \textit{deterministic} mechanism that enforces contractivity on the complementary space. Couplings are then designed by incorporating system controls in such a way that exploit this intrinsic property in a convenient way \cite{DebusscheOdasso2005, Odasso2006, HairerMattingly2006, KuksinShirikyan2000, Shirikyan2008, HairerMattinglyScheutzow2011,  GlattHoltzMattinglyRichards2017, ButkovskyKulikScheutzow2019, GlattHoltzMartinezRichards2021, GlattHoltzMartinezNguyen2022, FerrarioZanella2023, CarigiBrockerKuna2023, BrzezniakFerrarioZanella2023, Nguyen2023}. {The framework introduced in the present article ultimately takes its inspiration in the design of such couplings.}

{We establish here} a precise and rigorous relation between the existence of finitely many determining modes and the ability of the Olson-Titi (OT) and Azouani-Olson-Titi (AOT) continuous data assimilation (CDA) algorithms to converge in the paradigmatic context of the 2D Navier-Stokes {system.} {These} data assimilation algorithms are in fact closely related to couplings that have been designed in the study of the problem of unique ergodicity for many stochastically-driven equations in hydrodynamics.  {Indeed, the (AOT) algorithm} implements a control that has been fruitfully exploited in studying the problem of unique ergodicity for stochastically forced systems \cite{KuksinShirikyan2012, GlattHoltzMattinglyRichards2017, ButkovskyKulikScheutzow2019, GlattHoltzMartinezRichards2021, FerrarioZanella2023, BrzezniakFerrarioZanella2023, Nguyen2023}, {while the (OT) algorithm has effectively been exploited for these purposes in, for instance, \cite{DebusscheOdasso2005, Odasso2006}.}

Motivated by {the ideas behind such coupling designs, we introduce} a new, but closely related concept of \textit{intertwinement}, and show that in fact a stronger property holds, which is the existence of what we call \textit{self-synchronous intertwinements}. From this property, one can deduce both the finite determining modes property and reconstruction property of two particular data assimilation algorithms as direct consequences. By doing so, we develop a unifying theory in which determining modes and synchronization of continuous data assimilation can be studied simultaneously. From this point of view, we may consider this work to be a rigorous expansion to the context of nonlinear PDEs on the seminal works of L.M. Pecora and T.K. Carroll \cite{PecoraCarroll1990}, in which the phenomenon of \textit{complete synchronization} was discovered in the context of the Lorenz 63 via a version of what we refer to as the ``direct-replacement algorithm" (see \eqref{def:sync:q}), and on G.S. Duane, J.J. Tribbia, and J.B. Weiss \cite{DuaneTribbiaWeiss2006}, in which what we refer to as the ``nudging algorithm" (see \eqref{def:nudge:v}) was viewed from the perspective of synchronization for the purposes of continuous data assimilation \cite{DuaneTribbiaWeiss2006}.

Ultimately, due to the close relation of these CDA algorithms to the couplings that have been constructed in the {stochastic processes} literature, it is our hope that these ideas will also eventually find application in the study of stochastically forced systems. {Nevertheless, one immediate application of our results is the possibility to probe the number of determining modes via continuous data assimilation. The first systematic computational study to estimate the dimension of the minimal subspace of determining modes were carried out by E. Olson and E.S. Titi in \cite{OlsonTiti2008}. One of the main consequences of \cref{thm:main} is that essentially any ``CDA algorithm" (as defined by \cref{def:cda}) that derives from a self-synchronous intertwinement can be used to provide an upper bound estimate on this minimal dimension. Whether the upper bound provided is sharp remains a compelling issue of investigation. Lastly, many works have previously studied the problem of parameter estimation in dynamical systems from the perspective of synchronization in \textit{finite-dimensional settings} \cite{CrevelingJeanneAbarbanel2008, AbarbanelCrevelingFarsianKostuk2009, QuinnBryantCrevelingKleinAbarbanel2009, AbarbanelShirmanBreenKadakiaReyArmstrongMargoliash2017, CarlsonHudsonLariosMartinezNgWhitehead2021}. It is our hope that the rigorous development of this perspective in the context of an infinite-dimensional system will prove to be consequential for its further development in the context of parameter estimation in PDEs.}

In the remainder of the introduction, we hold an informal discussion of the central ideas {and defer making} rigorous statements for later on.

\subsection{Intertwinements: The Heart of the Matter}

The primary motivation of this paper is to establish a direct connection between the existence of determining modes for the 2D NSE and the convergence of two particular filtering algorithms for continuous data assimilation. As previously mentioned, it is a folklore that the two are intimately related due to the similarity in the proofs of these respective properties. Nevertheless, answers to basic questions regarding whether one property implies the other have remained largely unaddressed {with the sole exception of \cite{BiswasHudsonLariosPei17}}. This paper is an attempt to give {definitive} clarity to this issue.

The two algorithms of interest are the \textit{direct-replacement filter} and the \textit{nudging filter}. Consider a solution $u$ to \eqref{eq:nse}. Let $Qu=Q_Nu$, where $Q_N=I-P_N$. Given $N>0$, the direct-replacement filter is defined as the function $v$, which satisfies:
    \begin{align}\label{def:sync:q}
        \begin{split}
        v&=P_Nu+Q_Nv,\\
        \bdy_tQ_Nv+Q_N[(P_Nu+Q_Nv)\cdotp\nabla(P_Nu+Q_Nv)]&=-Q_N\nabla r+\nu \De Q_Nv+Q_Nf,\quad \nabla\cdotp Q_Nv=0
        \\
        Q_Nv(0)&=Q_Nv_0,
        \end{split}
    \end{align}
where $r$ denotes the scalar pressure field. Note that $q_0$ does not necessarily equal $Qu_0$. On the other hand, the nudging filter is defined as the function $\tv$, which satisfies
    \begin{align}\label{def:nudge:v}
        \bdy_t\tv+(\tv\cdotp\nabla)\tv=-\nabla\rt+\nu\De\tv+f-\mu P_N\tv+\mu P_Nu,\quad\nabla\cdotp \tv=0,\quad \tv(0)=\tv_0,
    \end{align}
where $\rt$ denotes the scalar pressure field, and not equal to $r$ from \eqref{def:sync:q} in general. Similarly, $\tv_0$ does not necessarily satisfy $P_N\tv_0=P_Nu_0$, and may be taken to be arbitrary. In \eqref{def:sync:q} and \eqref{def:nudge:v} are proposed as solutions to the problem of recovering the unobserved component, $Qu$, of the underlying 2D NSE system. ``Convergence" of these algorithms refer to the property that $v$ and $\tv$ \textit{synchronize} with $u$. Informally, this means that there exist appropriate choices of $N$ and $\mu$ such that
    \begin{align}\label{eq:synchronized}
        \lim_{t\goesto\infty}|v(t)-u(t)|=\lim_{t\goesto\infty}|\tv(t)-u(t)|=0.
    \end{align}
Note that in \eqref{eq:synchronized} (and below), the norm $|\cdotp|$ may be interpreted as the $L^2$-norm; this will be defined more precisely in \cref{sect:preliminaries}. Morally speaking, synchronization of \eqref{def:sync:q} and \eqref{def:nudge:v} with \eqref{eq:nse} is a concrete manifestation of the determining modes property since it not only asserts that knowledge of a sufficiently many modes, $p$, is enough to asymptotically determine the unobserved modes, $Qu$, but moreover furnishes an explicit approximation of the unobserved modes, $q\approx Qu\approx Q_N\tv$. It is therefore natural to expect that this reconstruction property is directly relatable to the existence of finitely many determining modes for \eqref{eq:nse}. {In what follows, we present an argument that uncovers a path for establishing this relation.}

Suppose that either one of the algorithms, (OT) or (AOT), above possesses the reconstruction property \eqref{eq:synchronized}. To show that this property implies that \eqref{eq:nse} has the determining modes property, i.e., has finitely many determining modes, one must establish, for any pair of solutions $u_1,u_2$ corresponding to external forces $f_1,f_2$, the existence of a cut-off, $N$, with the following property:
    \begin{align}\label{eq:det:modes}
        \lim_{t\goesto\infty}|P_Nu_1(t)-P_Nu_2(t)|=\lim_{t\goesto\infty}|f_1(t)-f_2(t)|=0\quad\text{implies}\quad \lim_{t\goesto\infty}|u_1(t)-u_2(t)|=0.
    \end{align}
Let $(v_1,v_2)$ be a pair satisfying either \eqref{def:sync:q} or \eqref{def:nudge:v} respectively corresponding to forces $(f_1,f_2)$. Then, by assumption, for $N$ sufficiently large, $v_j$ synchronizes with $u_j$. By the triangle inequality, one has
    \begin{align}\label{eq:control:argument}
        |u_1-u_2|\leq|u_1-v_1|+|v_1-v_2|+|v_2-u_2|,
    \end{align}
so that by the reconstruction property, one obtains
    \begin{align}\notag
        \lim_{t\goesto\infty}|u_1(t)-u_2(t)|&\leq\limsup_{t\goesto\infty}|v_1(t)-v_2(t)|.
    \end{align}
On the other hand
    \begin{align}\label{eq:v1v2:triangle}
        |P_Nv_1-P_Nv_2|\leq |P_Nv_1-P_Nu_1|+|P_Nu_1-P_Nu_2|+|P_Nu_2-P_Nv_2|,
    \end{align}
Then by a second application of the reconstruction property, which guarantees $|P_Nv_j-P_Nu_j|\goesto0$, in addition to the determining modes hypothesis \eqref{eq:det:modes}, namely that $|P_Nu_1-P_Nu_2|\goesto0$, one deduces from \eqref{eq:v1v2:triangle} that
    \begin{align}\notag
        \lim_{t\goesto\infty}|P_Nv_1(t)-P_Nv_2(t)|=0.
    \end{align}
Therefore, if one can guarantee that \eqref{def:sync:q} or \eqref{def:nudge:v} themselves possesses a \textit{determining modes-type property} that allows one to deduce that $|v_1(t)-v_2(t)|\goesto0$ as $t\goesto\infty$ from the fact that $|P_Nv_1(t)-P_Nv_2(t)|\goesto0$, as $t\goesto\infty$, then one may conclude from the argument above that $|u_1(t)-u_2(t)|\goesto0$.

A natural temptation is to attempt to deduce the determining modes property of \eqref{def:sync:q} or \eqref{def:nudge:v} \textit{as a consequence} of the determining modes property for \eqref{eq:nse} by a careful selection of $f_1, f_2$. While this indeed is the main mechanism at play, it is not rigorously possible to do so since the choice one must make requires $f_1, f_2$ to be \textit{state-dependent}. For instance, one may view a pair of solutions $(\tv_1,\tv_2)$ to \eqref{def:nudge:v} corresponding to forces $(\tf_1,\tf_2)$, as a pair of solutions to \eqref{eq:nse} corresponding to forces $(f_1,f_2)=(\tf_1-\mu P_N\tv_1+\mu P_Nu_1,\tf_2-\mu P_N\tv_2+\mu P_Nu_2)$. Thus, $f_j=f_j(\tv_j)$ and the {corresponding equation is no longer the same as \eqref{eq:nse}, but rather a \textit{perturbation} of \eqref{eq:nse}.} The obstruction is now clear and the problem reduces to addressing the following question: What class of perturbations to \eqref{eq:nse} allow one to formulate and establish a determining modes-type property? {The main} purpose of the concept of \textit{intertwinement} {is to formalize the class of such perturbations in such a way that simultaneously subsumes the convergence properties of existing CDA algorithms}.

The key idea is to view the determining modes property of \eqref{eq:nse} as a property of the \textit{augmented system} in which the pair $(u_1,u_2)$ simultaneously satisfies \eqref{eq:nse} corresponding to forcing $(f_1,f_2)$, respectively:
    \begin{align}\label{eq:independent:coupling}
        \begin{split}
        \bdy_tu_1+(u_1\cdotp\nabla) u_1&=-\nabla p_1+\nu\De u_1+f_1,\quad \nabla\cdotp u_1=0
         \\
        \bdy_tu_2+(u_2\cdotp\nabla) u_2&=-\nabla p_2+\nu\De u_2+f_2,\quad \nabla\cdotp u_2=0.
        \end{split}
    \end{align}
In the context of stochastic forcing, $f_j\sim dW_j$, such an augmented system is at once recognized as the \textit{independent coupling}. On the other hand, \eqref{def:sync:q} can be viewed as a coupled system in $(u,v)$ via:
    \begin{align}\label{eq:exact:coupling}
        \begin{split}
        \bdy_tu+(u\cdotp\nabla)u&=-\nabla p+\nu \De u+f,\quad \hspace{108.5pt}\nabla\cdotp u=0
        \\
        \bdy_tv+(v\cdotp\nabla)v&=-\nabla q+\nu \De v+f+{P_N((v\cdotp\nabla)v)}-{P_N((u\cdotp\nabla)u)},\quad \nabla\cdotp v=0.
        \end{split}
    \end{align}
In the context of stochastic forcing, \eqref{eq:exact:coupling} is analogous to an \textit{finite-dimensionally-exact coupling}, since the control term, ${P_N((v\cdotp\nabla)v)-P_N((u\cdotp\nabla)u)}$ serves to enforce $P_Nv=P_Nu$ exactly. Lastly, \eqref{def:nudge:v} can also be viewed as a coupled system in $(u,\tv)$ as:
    \begin{align}\label{eq:nudge:coupling}
        \begin{split}
        \bdy_tu+(u\cdotp\nabla)u&=-\nabla p+\nu \De u+f,\quad \hspace{74.5pt}\nabla\cdotp u=0
        \\
        \bdy_t\tv+(\tv\cdotp\nabla)\tv&=-\nabla q+\nu \De \tv+f-\mu P_N\tv+\mu P_Nu,\quad \nabla\cdotp v=0.
        \end{split}
    \end{align}
In the context of stochastic forcing, the term $f-P_Nv+P_Nu$ may often be viewed as a Girsanov shift of underlying stochastic forcing $f$, and \eqref{eq:nudge:coupling} has been exploited as an \textit{asymptotic coupling}. 

On the other hand, {taking stock of the heuristic argument made above, if we instead take} two copies,  $(v_1,v_2)$, of the second component of \eqref{eq:exact:coupling}, and two copies, $(\tv_1,\tv_2)$, of the second component of \eqref{eq:nudge:coupling}, respectively corresponding to the pair of forces $(f_1,f_2)$, one has:
    \begin{align}\label{eq:sync:pair}
        \begin{split}
        \bdy_tv_1+(v_1\cdotp\nabla)v_1&=-\nabla q_1+\nu \De v_1+\tf_1+P_N(v_1\cdotp\nabla)v_1,\quad \nabla\cdotp v_1=0.
        \\
        \bdy_tv_2+(v_2\cdotp\nabla)v_2&=-\nabla q_2+\nu \De v_2+\tf_2+P_N(v_2\cdotp\nabla)v_2,\quad \nabla\cdotp v_2=0,
        \end{split}
    \end{align}
where $\tf_j=f_j-P_N((u_j\cdotp\nabla)u_j)$, and 
 \begin{align}\label{eq:nudge:pair}
        \begin{split}
        \bdy_t\tv_1+(\tv_1\cdotp\nabla)\tv_1&=-\nabla \tq_1+\nu \De \tv_1+\tf_1-\mu P_N\tv_1,\quad \nabla\cdotp \tv_1=0,
        \\
        \bdy_t\tv_2+(\tv_2\cdotp\nabla)\tv_2&=-\nabla q_2+\nu \De \tv_2+\tf_2-\mu P_N\tv_2,\quad \nabla\cdotp \tv_2=0,
        \end{split}
    \end{align}
where ${\tf_j}=f_j+\mu P_Nu_j$.

In each of the above {instances, \eqref{eq:exact:coupling}--\eqref{eq:nudge:pair},} a common structure is readily identified:
    \begin{align}\label{eq:gen:pair}
       \begin{split}
        \bdy_tv_1+Av_1+B(v_1,v_1)&=f_1+m_{11}F(v_1)+m_{12}F(v_2)
        \\
        \bdy_tv_2+Av_2+B(v_2,v_2)&=f_2+m_{21}F(v_1)+m_{22}F(v_2),
        \end{split} 
    \end{align}
where $M=(m_{ij})_{ij}\in \RR^{2\times 2}$ and $F$ is some operator. Indeed, in \eqref{eq:exact:coupling}, $M$ is non-zero only in the second row, and $F(v)=P_N\PP(v\cdotp\nabla v)$, where $\PP$ denotes projection onto divergence-free vector fields (see \cref{sect:preliminaries}); in \eqref{eq:nudge:coupling}, $M$ is non-zero only in the second row as well, and $F(v)=P_Nv$; in \eqref{eq:sync:pair}, $M=\Id_{2\times2}$ is the identity matrix, and $F(v)=P_N\PP(v\cdotp\nabla)v$; in \eqref{eq:nudge:pair}, $M=\Id_{2\times2}$, and $F(v)=P_Nv$. Note that if we let $V=(v_1,v_2)$, $f=(f_1,f_2)$, and $F(V)=(F(v_1),F(v_2))$, then we can  rewrite \eqref{eq:gen:pair} more concisely in the following vector form: 
    \begin{align}\label{eq:gen:pair:vector}
       \begin{split}
        \bdy_tV+AV+B(V)&=f+MF(V).
        \end{split}
    \end{align}
We refer to \eqref{eq:gen:pair} as an \textit{intertwinement} of \eqref{eq:nse} {with $F$ denoting the \textit{intertwining function} and $M$ the \textit{intertwining matrix}.}

In this paper, we formulate a definition similar to the determining modes property, but in the generality of the {coupled system} \eqref{eq:gen:pair} (see \cref{def:intertwined}, \cref{def:fdss}). {Indeed, we observe that the property that $|u-v|\goesto0$ or $|u-\tv|\goesto0$ or $|v_1-v_2|\goesto0$ or $|\tv_1-\tv_2|\goesto0$ as $t\goesto\infty$ can all be viewed as the ability of the joint process defined by $U=(u,v)$, $(u,\tv)$, $(v_1,v_2)$, or $(\tv_1,\tv_2)$ to \textit{asymptotically self-synchronize} in the sense that $|\pi_1U-\pi_2U|\goesto0$ as $t\goesto\infty$, where $\pi_j$ denotes the projection on to coordinate $j=1,2$. In particular, the phenomenon that \eqref{eq:exact:coupling} or \eqref{eq:nudge:coupling} can be used to solve the state reconstruction problem in continuous data assimilation is realized as \textit{a special case} of self-synchronization within the intertwinement framework. 
Conversely, the ``reverse implication," that state reconstruction implies the existence of finitely many determining modes for the 2D NSE, can also be realized as a consequence of the self-synchronization property of the joint process. From this point of view, the ability of CDA algorithms to solve the state reconstruction is in some sense \textit{equivalent} to the finite determining modes property, as long as one views the finite determining modes property as a type of synchronization phenomenon. Within our framework, we are therefore able to achieve a conceptual clarity between the determining modes property of the 2D NSE and the ability of CDA algorithms to solve the state reconstruction problem that has hitherto been missing from the literature surrounding these ideas. 

To demonstrate the fruitfulness of our general definition, we propose} two classes of examples of $M$ and $F$ for which \eqref{eq:gen:pair} satisfies a self-synchronization-type property (see \cref{sect:examples}). These examples realize the direct-replacement filter \eqref{def:sync:q} and the nudging filter \eqref{def:nudge:v} as special cases. The rigorous analysis surrounding these examples are developed in a companion paper \cite{CarlsonFarhatMartinezVictor2025b}. We emphasize that the present article focuses on introducing the framework and establishing general results within this framework.

Ultimately, we view \cref{def:fdss} as a sublimation of the concept of the determining modes property {(see \cref{def:dm} and \cref{thm:Foias:Prodi})}, originally formulated for a given system, {e.g., \eqref{eq:nse},} into a property formulated for a {``lifting"} of the system (into a product space), which is induced by coupling the system to itself in a particular way, {e.g., \eqref{eq:gen:pair}.} In the {``lifted phase space,"} the determining modes property therefore becomes a statement about the {lifted system's} ability to asymptotically self-{synchronize. The concept of intertwinement and self-synchronization is} formally introduced in \cref{sect:intertwine}. {We develop a connection between intertwinements and continuous data assimilation in \cref{sect:cda}, where we view state reconstruction as a particular type of self-synchronization phenomenon. We develop several general results that elucidate how intertwinement, determining modes, and the state reconstruction property of certain CDA filters are related in \cref{sect:cda}. The main result in this regard is \cref{thm:main}, where we establish the implication that the state reconstruction property of CDA algorithms implies the existence of finitely many determining modes for the Navier-Stokes process \textit{as a consequence} of the self-synchronization property of intertwinements. We conclude \cref{sect:intertwine} by introducing several examples of intertwinements in \cref{sect:examples}. These examples demonstrate that the intertwinement framework is a legitimate generalization of the classical determining modes theory and the recent theories developed for studying state reconstruction in various CDA algorithms.}

Before proceeding to these sections, we provide the relevant mathematical preliminaries in \cref{sect:preliminaries}. {The paper concludes in}  \cref{sect:numerical} with a series of computational results that corroborate the theoretical results {and showcase the generality of the framework}.

\section{Mathematical Preliminaries}\label{sect:preliminaries}

{We let $H$ denote the space of $L^2$ real-valued vector fields, which are $2\pi$-periodic in each direction, divergence-free, and mean-free over $\Om=[0,2\pi]$, in the sense of distribution}. We let $\PP$ denote the Leray projection. Observe that $\PP H=H$. {We let $V$ denote the subspace of $H$ endowed with the $V$ topology}. We make use of the following notation for the inner products and norms on $H$ and $V$, respectively:
    \begin{align}\label{def:H}
        \lp u, v\rp=\int_\Om u(x)\cdotp v(x)dx,\quad |u|^2=\lp u,u\rp,
    \end{align}
and
    \begin{align}\label{def:V}
        \lpp u,v\rpp=\sum_{j=1,2}\int_\Om \bdy_ju(x)\cdotp\bdy_jv(x)\ dx,\quad \lVert u\rVert^2=\lpp u, u\rpp.
    \end{align}
The dual spaces of $H, V$ will be denoted by $H^*, V^*$ respectively. Then we have the following continuous imbeddings
    \begin{align}\notag
        V\subset H\subset H^*\subset V^*.
    \end{align}
In particular, we have the Poincar\'e inequality
    \begin{align}\label{est:Poincare}
        |u|\leq \lVert u\rVert,
    \end{align}
for all $u\in V$. For each $1\leq p\leq\infty$, we will also make use of the Lebesgue spaces, $L^p$, which denote the space of $p$-absolutely integrable functions endowed with the following norm:
    \begin{align}\label{def:Lp}
        \Abs{p}{u}=\lp\int_\Om|u(x)|^p dx\rp^{1/p},
    \end{align}
with the usual modification when $p=\infty$. For convenience, we will abuse notation and use the same notation to denote the corresponding subspace of $p$-absolutely integrable functions over $\Om$, which are mean-free and $2\pi$-periodic in each direction. It will be convenient to abuse notation further and not distinguish between the elements of $L^p$ being scalar functions or vector fields. 

Lastly, we denote the Stokes operator by $A=-\PP\De$ and define, for each $n\geq0$, integer powers, $A^{n/2}$, of $A$ by
    \begin{align}\label{def:A}
        A^{n/2}u=\sum_{k\in\ZZ^2\smod\{(0,0)\}}|k|^n\hat{u}_kw_k,\quad w_k(x)=\exp(ik\cdotp x).
    \end{align}
Then the domain, $D(A^{n/2})$, of $A^{n/2}$ is a subspace of $H$ endowed with the topology induced by
    \begin{align}\label{def:Hn}
        \no{n}{u}=|A^{n/2}u|=\lp\sum_{k\in\ZZ^2}|k|^{2n}|\hat{u}_k|^2\rp^{1/2}.
    \end{align}
Observe that
    \begin{align}\notag
        |u|=\Abs{0}{u},\qquad \rVert u\rVert=\no{0}{u}=\Abs{0}{A^{1/2}u}.
    \end{align}
    
Our analysis will make use of the Ladyzhenska and Agmon, respectively, interpolation inequalities: there exist absolute constants $C_L, C_A>0$ such that
	\begin{align}\label{est:interpolation}
		\Abs{4}{u}^2\leq C_L\lVert u\rVert|u|,\qquad
		\Abs{\infty}{u}^2\leq C_A|Au||u|.
	\end{align}
Another useful interpolation inequality, is the following:
	\begin{align}\label{est:interpolation:CS}
		\lVert u\rVert^2\leq |Au||u|.
	\end{align}
We will also make use of the Bernstein inequality: let $P_N$ denote projection onto Fourier wavenumbers, $|k|\leq N$, where $N>0$ is a real number. Denote the complementary projection by
    \begin{align}\label{def:QN}
        Q_N:=I-P_N.
    \end{align}
Then for any integers $m\leq n$
	\begin{align}\label{est:Bernstein}
		\no{n}{P_Nu}\leq N^{n-m}\no{m}{P_Nu},\quad \no{m}{Q_Nu}\leq N^{m-n}\no{n}{Q_Nu}.
	\end{align}
Observe that we also have the following borderline Sobolev inequality
    \begin{align}\label{est:Sobolev}
        |P_Nu|_\infty\leq C_S(\ln N)^{1/2}\|P_Nu\|.
    \end{align}

Given $f\in L^\infty(0,\infty;H)$, the \textit{generalized Grashof number} is defined as
    \begin{align}\label{def:Grashof}
        \Gr:=\frac{\sup_{t\geq0}|f(t)|}{\nu^2}.
    \end{align}
If $f\in L^\infty(0,\infty; D(A^{n/2}))$, then for each integer $n\geq1$, we define the shape factors of $f$ by
    \begin{align}\label{def:shape}
        \si_n:=\frac{\sup_{t\geq0}|A^{n/2}f(t)|}{|f|}.
    \end{align}

We will rewrite \eqref{eq:nse} in its functional form:
    \begin{align}\label{eq:nse:ff}
        \frac{du}{dt}+\nu Au+B(u,u)=f,\quad u(0)=u_0,
    \end{align}
where 
    \begin{align}\label{def:B}
        B(u,v):=\PP((u\cdotp\nabla)v).
    \end{align}
We also have the well-known, skew-symmetric property of $\lp B(u,v),w\rp$:
	\begin{align}\label{eq:B:skew}
		\lp B(u,v),w\rp=-\lp B(u,w),v\rp,
	\end{align}
for $u,v,w\in V$, which immediately implies
    \[
        \lp B(u,v),v\rp=0.    
    \]
We will also make use of the identity 
    \begin{align}\label{eq:B:enstrophy}
        \lp B(u,u),Au\rp=0.
    \end{align}
Observe that $B:D(A)\times V\goesto H$ via 
    \begin{align}\label{est:B:ext:H}
        |B(u,v)|\leq C_A^{1/2}|Au|^{1/2}|u|^{1/2}\|v\|.
    \end{align}
whenever $u \in D(A)$ and $v\in V$. Moreover, $B$ is also continuous as a bilinear mapping $B:V\times V\goesto V'$ via
    \begin{align}\label{est:B:ext}
        |\lp B(u,v),w\rp|\leq C_L\lVert u\rVert^{1/2}|u|^{1/2}\lVert v\rVert\|w\|^{1/2}|w|^{1/2},
    \end{align}
where $u,v\in V$ and $w\in V$, and $C_L$ is the constant appearing in \eqref{est:interpolation}.  The Frech\'et derivative of $B$ will be denoted by $DB$. Recall that $DB$ is defined by 
	\begin{align}\label{def:DB}
		DB(u)v=B(u,v)+B(v,u).
	\end{align}
By \eqref{est:B:ext:H}, it follows that $DB: D(A)\goesto L(D(A),H)$, $u\mapsto DB(u)$, while \eqref{est:B:ext} implies $DB: V\goesto L(V,V')$, where $L(X,Y)$ denotes the space of bounded linear operators mapping $X$ to $Y$.

We recall the following classical global existence and uniqueness result.

\begin{Thm}\label{prop:nse:ball}
Given $f\in L^\infty(0,\infty;H)$, $u_0\in V$, and $T>0$, there exists a unique solution $u\in C([0,T];V)\cap L^2(0,T;D(A))$ such that $u(0)=u_0$. Moreover, there exists $t_0=t_0(\|u_0\|,|f|)$ such that
    \begin{align}\label{est:absorb:L2}
        \sup_{t\geq t_0}|u(t)|\leq 2\nu\si_{-1}\Gr=:2\rho_0,\qquad \sup_{t\geq t_0}\|u(t)\|\leq 2\nu \Gr=:2\rho_1.
    \end{align}
{Moreover, if we let $u(\cdotp;u_0,f)$ denote the unique global-in-time solution of \eqref{eq:nse:ff}, then the balls $B_H(\rho_0)$, $B_V(\rho_1)$ are forward-invariant sets and $B_H(2\rho_0)$, $B_V(2\rho_1)$ are forward-absorbing for the process $\{S(t,s;f):t\geq s\}$ induced by $u(\cdotp;u_0,f)$, i.e., $S_{(F,M)}(t,s;f)u_0=u(t;u_0,\tau_sf)$, $u(s;u_0,\tau_sf)=u_0$, where $\tau_sf(t)=f(t+s)$, satisfy
    \begin{align}\notag
        S(t,0;f)u_0\in B_H(\rho_0)\ (\text{resp.}\ B_V(\rho_1)),\quad \text{for all}\ u_0\in B_H(\rho_0)\ (\text{resp.}\ B_V(\rho_1)),
    \end{align}
and, for any $u_0$ belonging to a bounded set in $H$ (resp. $V$), there exists $t_0>0$ such that
    \begin{align}\notag
         S(t,0;f)u_0\in B_H(2\rho_0)\ (\text{resp.}\ B_V(2\rho_1)),\quad \text{for all}\ t\geq t_0.
    \end{align}
}
\end{Thm}

We will refer to the solutions guaranteed by \cref{prop:nse:ball} as \textit{strong solutions} {and the corresponding process induced by the global well-posedness of its initial value problem as the \textit{Navier-Stokes process}}. We note that the forward-invariance (and forward-absorbing property) of $B_H(\rho_0)$ and $B_V(\rho_1)$ follow from the elementary inequalities which hold for strong solutions of \eqref{eq:nse:ff}:
    \begin{align}\label{est:abs:ball}
        \begin{split}
        |u(t)|^2&\leq e^{-\nu t}|u_0|^2+\rho_0^2(1-e^{-\nu t}),
        \\
        \lVert u(t)\rVert^2&\leq e^{-\nu t}\lVert u_0\rVert^2+\rho_1^2(1-e^{-\nu t}),
        \end{split}
    \end{align}
for all $t\geq0$ and $u_0\in V$, and let $p=P_Ku$ denote the projection of $u$ onto the wave-numbers $|k|\leq K$.

We will also make use of the global well-posedness (in the sense implied by  \cref{def:gwp}) of the corresponding initial value problems for direct-replacement filter and the nudging system, which were developed in \cite{OlsonTiti2003} and \cite{AzouaniOlsonTiti2014}, respectively. We state them here for the sake of completeness. For both statements, given $f\in L_{loc}^\infty(0,\infty;H)$ and $u_0\in V$, we let $u$ denote the unique global-in-time solution to \eqref{eq:nse:ff} such that $u\in C([0,T];V)\cap L^2(0,T;D(A))$ and $\frac{du}{dt}\in L^2(0,T;H)$, for all $T>0$.

\begin{Thm}[Theorem 3.1, \cite{OlsonTiti2003}]\label{thm:OT}
 For any $K>0$ and $q_0\in V$ such that $Q_Nq_0=q_0$, there exists a unique $q$ such that $q\in C([0,T];V)\cap L^2(0,T;D(A))$, $\frac{dq}{dt}\in L^2(0,T;H)$, for all $T>0$, and satisfies
    \begin{align}\label{eq:sync}
        \frac{dq}{dt}+\nu Aq+Q_KB(p+q,p+q)=Q_Kf,\quad q(0)=q_0.
    \end{align}
In particular, for $v=P_Ku+q$, the pair $(u,v)$ equivalently satisfies the following system of equations:
    \begin{align}\label{eq:sync:OT}
        \begin{split}
        \frac{du}{dt}+\nu Au+B(u,u)&=f,\quad u(0)=u_0\\
        \frac{dv}{dt}+\nu Av+B(v,v)&=f+P_K\left(B(v,v)-B(u,u)\right),\quad v(0)=P_Ku_0+q_0.
        \end{split}
    \end{align}
\end{Thm}

\begin{Thm}[Theorem 6, \cite{AzouaniOlsonTiti2014}]\label{thm:AOT}
For any $K>0$ and $\tv_0\in V$, there exists a unique $\tv$ such that $\tv\in C([0,T];V)\cap L^2(0,T;D(A))$, $\frac{d\tv}{dt}\in L^2(0,T;H)$, for all $T>0$, and satisfies
    \begin{align}\label{eq:nudge}
        \frac{d\tv}{dt}+\nu A\tv+B(\tv,\tv)=f-\mu P_K\tv+\mu P_Ku,\quad \tv(0)=\tv_0.
    \end{align}
In particular, the pair $(u,\tv)$ satisfies the following system of equations:
    \begin{align}\label{eq:nudge:AOT}
        \begin{split}
        \frac{du}{dt}+\nu Au+B(u,u)&=f,\quad u(0)=u_0\\
        \frac{d\tv}{dt}+\nu A\tv+B(\tv,\tv)&=f-\mu P_K\tv+\mu P_Ku,\quad \tv(0)=\tv_0.
        \end{split}
    \end{align}
\end{Thm}

We conclude this section {with} an elementary {Gr\"onwall-type lemma which we will make use of later}.

\begin{Lem}\label{lem:gronwall:decay}
Let $z_1,z_2,z_3:(0,\infty)\goesto[0,\infty)$ be given such that $\lim_{t\goesto\infty}z_j(t)=0$, for $j=1,2$. Suppose $x:[0,\infty)\goesto[0,\infty)$ is a differentiable function such that
	\begin{align}\notag
		x'+\al x+\be y \leq z_1+z_2x+z_3y,
	\end{align}
holds for for all $t>0$, for some $\al, \be>0$, and some dominating function $y:[0,\infty)\goesto[0,\infty)$, i.e., $x\leq y$. Then $\lim_{t\goesto\infty}x(t)=0$.
\end{Lem}

\begin{proof}
Since $z_j(t)\goesto0$ as $t\goesto\infty$, for $j=1,2,3$, given $0<\eps<1$, there exists $t_0>0$ such that $z_1(t)\leq (\al+\be)\eps(1-\eps)$ and $z_2(t)\leq \al\eps$, $z_3(t)\leq \be\eps$, for all $t\geq t_0$. Then
    \begin{align}
        x'+\al x+\be y&\leq (\al+\be)\eps(1-\eps)+\al \eps x+\be\eps y\notag,
    \end{align}
for all $t\geq t_0$. In particular
    \begin{align}\notag
        x'+\al(1-\eps) x+\be(1-\eps)y&\leq (\al+\be) \eps(1-\eps).
    \end{align}
Since $x\leq y$, it follows that
    \begin{align}\notag
        x'+(\al+\be)(1-\eps)x\leq (\al+\be)\eps(1-\eps).
    \end{align}
Hence
    \begin{align}
        x(t)&\leq e^{-(\al+\be)(1-\eps)(t-t_0)}x(t_0)+\eps(1-e^{-(\al+\be)(1-\eps)(t-t_0)})\notag
        \\
        &=e^{-(\al+\be)(1-\eps)(t-t_0)}(x(t_0)-\eps)+\eps.\notag
     \end{align}
Denote by $a_+(t)=\max\{a(t),0\}$.  We then choose $t_1\geq t_0$ such that 
    \begin{align}\notag
        t_1\geq  t_0+\frac{1}{(\al+\be)(1-\eps)}\ln\left(\frac{(x(t_0)-\eps)_+}{\eps}\right).
    \end{align}
Thus for all $t\geq t_1$
	\begin{align}\notag
		x(t)\leq 2\eps,
	\end{align}
as desired.
\end{proof}

\section{The Paradigm of Intertwinement and Self-Synchronization}\label{sect:intertwine}

{In this section, we introduce the concept of intertwinement and self-synchronization, as well as associated terminology for the intertwinement framework. Its main goal is to establish a general descriptive theory, in the particular setting of the 2D Navier-Stokes equations, in which the determining modes property can be located as a special case; of course, the framework developed here is intended to apply more generally to other dissipative systems. Once this has been done, we then conceptualize the notion of a continuous data assimilation algorithm within the intertwinement framework and proceed to develop rigorous statements between the ability of such algorithms to asymptotically reconstruct the unobserved state variables and self-synchronizing intertwinements. One of the main achievements of the intertwinement framework is to demonstrate a form of equivalence between the finite determining modes property of the 2D Navier-Stokes system and the asymptotic state reconstruction property of continuous data assimilation algorithms. Finally, we conclude the section by developing several explicit examples of intertwinements, thus demonstrating the richness of the theory and legitimacy of this perspective.

To begin, let us} recall the definition of determining modes for the 2D NSE, originally introduced by Foias and Prodi in \cite{FoiasProdi1967}.

\begin{Def}\label{def:dm}
Given $f_1,f_2\in L^\infty(0,\infty;H)$, let $u_1=u(\cdotp; u_0^1), u_2=u(\cdotp;u_0^2)$ denote the global-in-time unique strong solutions of the initial value problems
    \begin{align}\label{eq:nse:coupled}
        \begin{split}
        \frac{du_1}{dt}+\nu Au_1+B(u_1,u_1)&=f_1,\quad u_1(0)=u_0^1,
        \\
         \frac{du_2}{dt}+\nu Au_2+B(u_2,u_2)&=f_2,\quad u_2(0)=u_0^2.
         \end{split}
    \end{align}
We say that the {Navier-Stokes process} has the \textbf{finite determining modes property} if for all $f_1,f_2$, there exists a finite $N>0$ such that 
     \begin{align}\label{cond:dm}
         |P_Nu(t;u_0^1) - P_Nu(t;u_0^2)| \goesto 0\quad\text{and}\quad|f_1(t) - f_2(t)| \goesto 0,\quad \text{as}\ t\goesto\infty,
     \end{align}
implies
     \begin{align}\label{eq:dm}
         |u(t;u_0^1) - u(t;u_0^2)| \goesto 0,\quad\text{as}\ t\goesto\infty,
     \end{align}
for all $u_0^1,u_0^2\in V$; the smallest $N$, for a given $f_1,f_2$ above, is the \textbf{number of determining modes}, {which we will typically denote as $N_{dm}$}.
\end{Def}

The core idea is to expand \cref{def:dm} in a way that effectively allows $f_1, f_2$ to depend on $u_1,u_2$. We do so by introducing the notion of \textit{intertwinement}.
 
\begin{Def}\label{def:intertwined}
Let $g_1,g_2\in L^\infty(0,\infty;H)$ and ${F}:V\goesto H$ such that ${F(0)=0}$. Then the \textbf{intertwined Navier-Stokes system} is given by
    \begin{align}\label{eq:nse:intertwined}
        \begin{split}
        \frac{dv_1}{dt}+\nu Av_1+B(v_1,v_1)&=g_1+m_{11}{F}(v_1)+m_{12}{F}(v_2)
        \\
         \frac{dv_2}{dt}+\nu A{v}_2+B(v_2,v_2)&=g_2+m_{21}{F}(v_1)+m_{22}{F}(v_2),
         \end{split}
    \end{align}
for some $M=(m_{ij})_{i,j}\in\RR^{2\times2}$. 
{{We} refer to $F$ as an \textbf{intertwining function}, $M$ as an \textbf{intertwining matrix}, and the pair $(F,M)$ an \textbf{intertwining pair}. Finally, whenever $v=(v_1,v_2)$ satisfies \eqref{eq:nse:intertwined} for some intertwining pair, we call the triplet $(v,F,M)$ an \textbf{intertwinement}.}
\end{Def}

\begin{Def}\label{def:gwp}
{ 
Given $g_1,g_2\in L^\infty(0,\infty;H)$ and $v_0^1,v_0^2\in V$, we say that the associated initial value problem of \eqref{eq:nse:intertwined} is \textbf{globally well-posed} if there exists a unique pair $(v_1,v_2)$ such that for all $T>0$, it holds that $v_1,v_2\in C([0,T];V)\cap L^2(0,T;D(A))$ and satisfies \eqref{eq:nse:intertwined} for $t\in (0,T)$, with $v_j|_{t=0}=v_0^j$, for $j=1,2$. 
}
\end{Def}

{Given an intertwinement $(v,F,M)$ with globally well-posed initial value problem, we may subsequently denote by 
    \begin{align}\label{def:v:notation}
    v(\cdotp;v_0,g)=(v_1(\cdotp;v_0^1,g_1),v_2(\cdotp;v_0^2,g_2))
    \end{align}
the unique solution of \eqref{eq:nse:intertwined} corresponding to initial data $v|_{t=0}=v_0:=(v_0^1,v_0^2)$ and external force $g=(g_1,g_2)$. For the remainder of the section, it will be assumed that $(v,F,M)$ has globally well-posed initial value problem. We point out, however, that when we address specific examples of intertwinements, the issue of global well-posedness must be verified separately.
}

\begin{Def}\label{def:fdss}
An intertwinement is \textbf{finite-dimensionally-driven self-synchronous} if for all $g_1,g_2$, there exists a finite $N\geq0$ such that
	\begin{align}\label{cond:fdss}
	 |P_Nv_1(t,v_0^1,{g_1})-P_Nv_2(t;v_0^2,{g_2})|\goesto0\quad\text{and}\quad |g_1(t)-g_2(t)|\goesto0,\quad \text{as $t\goesto\infty$},
	 \end{align}
implies
	\begin{align}\label{eq:fdss}
		 |v_1(t;v_0^1,{g_1})-v_2(t;v_0^2,{g_2})|\goesto0,\quad\text{as $t\goesto\infty$},
	\end{align}
for all $v_0^1,v_0^2\in V$. {We denote the smallest such $N$, for a given $g_1,g_2$ as above, by $N_{ss}(g_1,g_2)$ and refer to the function $N_{ss}:=N_{ss}(g_1,g_2)$ as the \textbf{dimension of self-synchronization}.} 
\end{Def}

The special case of a zero-dimensionally-driven self-synchronous intertwinement, i.e., $N_{ss}\equiv0$, is of particular interest. 

\begin{Def}\label{def:fdss:zero}
We say that $(v,F,M)$ is \textbf{self-synchronous} if
	\begin{align}\notag
	   |g_1(t)-g_2(t)|\goesto0,\quad \text{as $t\goesto\infty$},
	 \end{align}
implies
	\begin{align}\notag
		 |v_1(t;v_0^1,{g_1})-v_2(t;v_0^2,{g_2})|\goesto0,\quad\text{as $t\goesto\infty$},
	\end{align}
for all $v_0^1,v_0^2\in V$.
\end{Def}

Throughout the paper, we will often refer to the properties of being finite-dimensionally-driven self-synchronous or self-synchronous as \textit{self-synchronization properties} or, simply \textit{synchronization properties}.

{With the above framework in hand, it is remarkable that an} example of {an} intertwinement is {provided when either $M\equiv0$ or} $F\equiv0$, i.e., without the assistance of coupling. {Note that in this case \eqref{eq:nse:intertwined} simply becomes \eqref{eq:nse:coupled}.} Indeed, this case {corresponds precisely to} the seminal result of Foias and Prodi  \cite{FoiasProdi1967}. {We refer to either of these special cases, $(v,0,M)$, $(v,F,0)$, as the} \textbf{\textit{trivial intertwinement}}. We may then equivalently state the classical result of Foias and Prodi in the following succinct manner.

\begin{Thm}[Existence of Determining Modes \cite{FoiasProdi1967}]\label{thm:Foias:Prodi}
The trivial intertwinement is finite-dimensionally-driven self-synchronous.
\end{Thm}

The fact that the trivial intertwinement is finite-dimensionally-driven self-synchronous is indicative of the fact that the low-modes and high-modes of the error between two solutions are sufficiently coupled \textit{via its nonlinearity} in order for low-mode synchronization to induce high-mode synchronization.


\begin{Thm}\label{thm:fdss}
Let $(v,F,M)$ be an intertwinement. Suppose that there exists $N_*>0$ and $t_*$ sufficiently large such that 
    \begin{align}\label{cond:coercivity}
        &\left(B(v_1,v_1)-B(v_2,v_2)-\left((m_{11}-m_{21})F(v_1)+(m_{12}-m_{22})F(v_2)\right),v_1-v_2\right)\notag
               \\
        &\geq -\left[\eps\nu+C_2(\|P_Nv_1-P_Nv_2\|_{m_2})\right]\|v_1-v_2\|^2-C_1(\|P_Nv_1-P_Nv_2\|_{m_1})|v_1-v_2|^2\notag
        \\
        &\quad-C_0(\|P_Nv_1-P_Nv_2\|_{m_0}),
    \end{align}
holds for all $t\geq t_*$, for all $N\geq N_*$, for some $\eps\in(0,1)$, some integers $m_0,m_1,m_2$, and some non-negative, non-decreasing functions $C_0(x),C_1(x), C_2(x)$, defined for $x\geq0$, such that $C_j(0)=0$ and $C_j$ are continuous at $0$, for $j=0,1,2$. Then  {$(v,F,M)$} is finite-dimensionally-driven self-synchronous.
\end{Thm}

\begin{proof}
Let $w=v_1-v_2$, $h=g_1-g_2$, and $p=P_Nw$, for some $N$ to be specified. Observe that $w$ satisfies the following equation:
    \begin{align}\label{eq:w:general}
        &\frac{dw}{dt}+\nu Aw
        \\
        &+\left(B(v_1,v_2)-B(v_2,v_2)\right)-\left((m_{11}-m_{21})F_1(v_1)+(m_{12}-m_{22})F_2(v_2)\right)=h.\notag
    \end{align}
Upon taking the inner product in $L^2$ or \eqref{eq:w:general} with $w$, we obtain
    \begin{align}
        &\frac{1}2\frac{d}{dt}|w|^2+\nu\|w\|^2\notag
        \\
        &+\left(\left(B(v_1,v_2)-B(v_2,v_2)\right)-\left((m_{11}-m_{21})F_1(v_1)+(m_{12}-m_{22})F_2(v_2)\right),w\right)=(h,w).\notag
    \end{align}
By the Cauchy-Schwarz inequality and Young's inequality, we have
    \begin{align}
        |(h,w)|&\leq \|h\|_*\|w\|\leq \frac{1}{2\nu}\|h\|_*^2+\frac{\nu}2\|w\|^2.\notag
    \end{align}
Combining this with \eqref{cond:coercivity}, we then deduce that there exists $N_*>0$ such that
    \begin{align}\label{eq:w:gen:structure}
        \frac{d}{dt}|w|^2&+2(1-\eps)\nu\|w\|^2\notag
        \\
        &\leq \frac{1}\nu\|h\|_*^2+2C_2(\|P_Nv_1-P_Nv_2\|_{m_2})\|w\|^2+2C_1(\|p\|_{m_1})|w|^2+2C_0(\|p\|_{m_0}),
    \end{align}
holds for all $N\geq N_*$ and all $t\geq t_0$, where $t_0$ is sufficiently large. By the Poincar\'e inequality, it follows that
    \begin{align}\notag
       \|h\|_*\leq |h|,\quad C_j(\|p\|_{m_j})\leq C_j(N^{m_j}|p|),\quad j=0,1,2
       \end{align}
Now, in order to show that \cref{def:fdss} holds, let us assume that
    \begin{align}\notag
        |p(t)|\goesto0,\quad |h(t)|\goesto0,\quad \text{as}\ t\goesto\infty,
    \end{align}
where $N\geq N_*$. From the assumed continuity at $0$, it follows that $C_j(\|p(t)\|_{m_j})\goesto0$, as $t\goesto\infty$. In particular, $C_2(\|p(t)\|_{m_2})\leq\frac{1-\eps}2$, for all sufficiently large $t$. We therefore conclude, upon applying \cref{lem:gronwall:decay} to \eqref{eq:w:gen:structure}, that $|w(t)|\goesto0$ as $t\goesto\infty$, as desired.
\end{proof}

We also identify a similar condition for ensuring self-synchronous intertwinements.

\begin{Thm}\label{thm:ss}
Let $(v,F,M)$ be an intertwinement. Suppose that there exists $N_*>0$ and $t_*$ sufficiently large such that 
    \begin{align}\label{cond:coercivity:ss}
        &\left(B(v_1,v_1)-B(v_2,v_2)-\left((m_{11}-m_{21})F(v_1)+(m_{12}-m_{22})F(v_2)\right),v_1-v_2\right)\notag
               \\
        &\geq -\eps\nu\|v_1-v_2\|^2-H(t),
    \end{align}
holds for all $t\geq t_*$, for all $N\geq N_*$, for some $\eps\in(0,1)$, and some non-negative function $H$ such that $H(t)\goesto0$ as $t\goesto\infty$. Then  {$(v,F,M)$} is self-synchronous.
\end{Thm}

\begin{proof}
By the Cauchy-Schwarz inequality and Young's inequality, we have
    \begin{align}
        |(h,w)|&\leq \|h\|_*\|w\|\leq \frac{1}{(1-\eps)\nu}\|h\|_*^2+\frac{1-\eps}{2}\nu\|w\|^2.\notag
    \end{align}
Combining this with \eqref{cond:coercivity:ss}, we then deduce that there exists $N_*>0$ such that
    \begin{align}\label{eq:w:gen:structure:ss}
        \frac{d}{dt}|w|^2&+(1-\eps)\nu\|w\|^2\leq \frac{1}{(1-\eps)\nu}\|h\|_*^2+2H,
    \end{align}
holds for all $N\geq N_*$ and all $t\geq t_*$, where $t_*$ is sufficiently large. By the Poincar\'e inequality, it follows that $\|h\|_*\leq |h|$. To show that $(v,F,M)$ is self-synchronous, suppose that $|h(t)|\goesto0$. Since $H(t)\goesto0$ as $t\goesto0$, therefore conclude, upon applying \cref{lem:gronwall:decay} to \eqref{eq:w:gen:structure:ss}, that $|w(t)|\goesto0$ as $t\goesto\infty$, as desired.
\end{proof}

{In the next section, we formalize the notion of a \textit{continuous data assimilation algorithm} within the intertwinement framework. This will allow us to rigorously study the connection between the finite determining modes property of the Navier-Stokes equation and the ability of continuous data assimilation algorithms to reconstruct unobserved state variables.
}

\subsection{Continuous Data Assimilation Algorithms}\label{sect:cda}

{
In order to define a ``continuous data assimilation algorithm" within the intertwinement framework, we propose the following additional terminology. 

\begin{Def}\label{def:sync}
Given $v_1,v_2\in L^\infty(0,\infty;H)$, we say that $(v_1,v_2)$ is an \textbf{synchronous pair} or simply, \textbf{synchronous}, if
    \begin{align}\label{cond:f:sync}
            |v_1(t) - v_2(t)| \goesto 0,\quad \text{as}\ t\goesto\infty.
    \end{align}
 We say that $(u_1,u_2)$ is a \textbf{Foias-Prodi pair} if there exists $N>0$ and synchronous pair $(f_1,f_2)$ such that $(P_Nu_1,P_Nu_2)$ is also a synchronous pair, where $(u_1,u_2)$ satisfies \eqref{eq:nse:coupled}; we refer to $N$ as the \textbf{scale} of the Foias-Prodi pair. 
 
 Given $F:V\goesto H$, we say that $F$ is \textbf{strongly Foias-Prodi compatible} if there exists $N>0$ such that    
 \begin{align}\label{eq:F:adapted}
       \lim_{t\goesto\infty}|F(u_1(t))-F(u_2(t))|=0,
    \end{align}
holds along any Foias-Prodi pair $(u_1,u_2)$ at scale $N$. If \eqref{eq:F:adapted} holds along synchronous Foias-Prodi pairs, then $F$ is \textbf{weakly Foias-Prodi compatible.}
\end{Def}

\begin{Rmk}\label{rmk:FP:pair:convergent}
It is clear that if $(u_1,u_2)$ is a synchronous Foias-Prodi pair, then it is a Foias-Prodi pair at any finite scale $N<\infty$. In particular, $(u_1,u_2)$ is a synchronous Foias-Prodi pair if and only if it is a Foias-Prodi pair at scale $N=\infty$. Note that we interpret $P_\infty=I$ as the identity operator $Iu=u$. 

Lastly, note that the $N$ for which $F$ is strongly Foias-Prodi compatible may be thought of as indicating the ``scale of compatibility." In contrast, when $F$ is weakly Foias-Prodi compatible, there is no such associated notion since weakly Foias-Prodi compatible functions are agnostic to the scale of the Foias-Prodi pairs that $F$ is evaluated along.
\end{Rmk}

With this terminology, \cref{def:dm} can be restated in several equivalent forms.

\begin{Cor}\label{prop:dm:equiv}
The following are equivalent:
	\begin{enumerate}
		\item The Navier-Stokes process has the finite-determining modes property;
		\item For every synchronous pair $(f_1,f_2)\in L^\infty(0,\infty;H)^2$, there exists $N>0$ such that every corresponding Foias-Prodi pair at scale $N$ is also synchronous;
		\item For every synchronous pair $(f_1,f_2)\in L^\infty(0,\infty;H)^2$, there exists $N>0$ such that $(u_1,u_2)$ is synchronous whenever $(P_Nu_1,P_Nu_2)$ is synchronous.
	\end{enumerate}
\end{Cor}

\cref{def:fdss} can also be restated as follows.

\begin{Cor}\label{prop:fdss:equiv}
An intertwinement is finite-dimensionally-driven self-synchronous if and only if for every synchronous pair $(g_1,g_2)\in L^\infty(0,\infty;H)^2$, there exists $N\geq0$, such that $(v_1,v_2)$ is synchronous whenever $(P_Nv_1,P_Nv_2)$ is synchronous.
\end{Cor}

\begin{Rmk}\label{rmk:FP:pair}
The notion of a Foias-Prodi pair above clarifies the role of the initial data $u_0^1,u_0^2$. In particular, it allows for the possibility that a Foias-Prodi pair at scale $N<N_{dm}$ be synchronous for some initial data. In general, a Foias-Prodi pair need not be synchronous. Of course, any Foias-Prodi pair at scale $N\geq N_{dm}$ is automatically synchronous, independently of the initial data. From this point of view, it is thus legitimate to introduce the notion of a Foias-Prodi pair and, subsequently, to distinguish between when it is synchronous or not.
\end{Rmk}

We are now ready to define our notion of a ``continuous data assimilation algorithm," {which considers intertwinements formed by certain \textit{uni-directional} couplings}.

\begin{Def}\label{def:cda}
Any intertwinement {$(v,F,M)$ such that} $F$ is Foias-Prodi compatible and $M$ belongs to
    \begin{align}\label{eq:M:cda}
    M\in\cU:=\left\{\begin{pmatrix}0 & 0\\ m_1 & m_2\end{pmatrix}, \begin{pmatrix}m_1 & m_2\\ 0 & 0\end{pmatrix}: m_1,m_2\neq0\right\},
    \end{align}
is a \textbf{continuous data assimilation algorithm}. We say that it is of \textbf{strong-type} or \textbf{weak-type} if $F$ is either strongly or weakly  Foias-Prodi compatible, respectively.

Given a continuous data assimilation algorithm, we say that it possesses the \textbf{reconstruction property at scale $N$} if for all $f\in L^\infty(0,\infty,H)$, there exists $N$ such that
    \begin{align}\label{def:reconstruction}
        \lim_{t\goesto\infty}|v_1(t;v_0^1, {f})-v_2(t;v_0^2,{f})|\goesto0,
    \end{align}
holds for all $(v_0^1,v_0^2)\in V\times V$.
\end{Def}

\begin{Rmk}\label{rmk:cda:matrix}
Due to the symmetry present in the coupling \eqref{eq:nse:intertwined}, it is actually sufficient to consider matrices of the form $\begin{pmatrix}m_1&m_2\\ 0&0\end{pmatrix}$, for instance, in defining the class $\cU$ in \cref{def:cda}. Indeed, one can simply swap the roles of $v_1$ and $v_2$. For our purposes, it will be expedient to include both types of matrices in \eqref{eq:M:cda}. Also note that as a result of the symmetry, we may unambiguously adopt the convention that whenever $\tilde{\cU}\subset\cU$, then $\tilde{\cU}$ satisfies the property that $\begin{pmatrix}m_1&m_2\\ 0&0\end{pmatrix}\in\tilde{\cU}$ for some $m_1,m_2\in\RR$ if and only if $\begin{pmatrix}0&0\\ m_1&m_2\end{pmatrix}$; we will then refer to matrices of the form $\begin{pmatrix}m_1&m_2\\ 0&0\end{pmatrix}\in\tilde{\cU}$ as $M_+$, to reflect the fact that the coupling comes from the ``top" of the matrix, and $\begin{pmatrix}0&0\\ m_1&m_2\end{pmatrix}$ as $M_-$, to reflect the fact that the coupling comes from the ``bottom" of the matrix.
\end{Rmk}

Note that for matrices $M\in\cU$ whose first row, for instance, is the zero row, then the first component $v_1$ of $v$ is simply a solution to the 2D NSE. In particular, \cref{def:cda} formalizes the idea of a continuous data assimilation algorithm as an intertwinement whose intertwining function is suitably adapted to the Navier-Stokes process and whose intertwining matrix uni-directionally couples this perturbed NS system to the original NS system. In the context of such intertwinements, the concept of 
\textit{self-synchronization} then manifests as the ability of the second component, $v_2$, to \textit{asymptotically reconstruct} its first component, $v_1$. The following lemma provides a simple class of operators $F$ that give way to CDA algorithms of either weak- or strong-type.

\begin{Lem}\label{lem:dm:cda}
Suppose that $(F,M)$ is an intertwining pair such that $M\in\cU$ and $F:V\goesto H$ satisfies
    \begin{align}\label{cond:F:cts}
    \lim_{u,v\in V,\ \|u-v\|\goesto0}|F(u)-F(v)|=0.
    \end{align}
Then $(v,F,M)$ is a continuous data assimilation algorithm of weak-type.

If $F$ additionally satisfies $F=F\circ P_K$, for some $K>0$, then $(v,F,M)$ defines a continuous data assimilation algorithm of strong-type at scale $K$.
\end{Lem}

\begin{proof}
Given $N>0$, let $(u_1,u_2)$ denote any synchronous Foias-Prodi pair. By \eqref{cond:F:cts}, it therefore follows that $\lim_{t\goesto\infty}|F(u_1(t))-F(u_2(t))|=0$.

On the other hand, given a Foias-Prodi pair at scale $N\geq K$, we have $\lim_{t\goesto\infty}|P_Nu_1(t)-P_Nu_2(t)|=0$. Since $N\geq K$, we see that $\lim_{t\goesto\infty}|P_Ku_1(t)-P_Ku_2(t)|=0$. By \eqref{cond:F:cts}, it therefore follows that $\lim_{t\goesto\infty}|F(u_1(t))-F(u_2(t))|=\lim_{t\goesto\infty}|F(P_Ku_1(t))-F(P_Ku_2(t))|=0$.

\end{proof}

\cref{lem:dm:cda} 
thus shows that an intertwining pair $(F,M)$ induces a continuous data assimilation algorithm  in the sense of \cref{def:cda} whenever $M\in \cU$ and $F$ satisfies a form of continuity that is  suitably adapted to the finite determining modes property of the Navier-Stokes system.

We are now in a position to state and prove the main theorem of this section, which addresses a rigorous relation between continuous data assimilation of \textit{strong-type} with their reconstruction property. Informally stated, we prove that whenever a continuous data assimilation algorithm of strong-type has the reconstruction property, then the underlying Navier-Stokes process has the finite determining modes property, as long as the continuous data assimilation algorithm derives from a sufficiently rich class of finite-dimensionally-driven self-synchronous intertwinements. To state this result, let us introduce the following class of matrices:
    \begin{align}\label{def:M0}
        \cD:=\left\{D=\begin{pmatrix} d_1&0\\ 0&d_2\end{pmatrix}:d_1,d_2\neq0\right\}.
    \end{align}
We will also represent elements $D\in\cD$ by $D=\diag(d_1,d_2)$. Also, recall that an operator $F:V\goesto H$ is \textit{locally bounded} if for each ball $B\subset V$, there exists $C$ such that $|F(u)|\leq C$, for all $u\in B$.

\begin{Thm}\label{thm:main}
Let $F:V\goesto H$ be given such that $F(0)=0$ and $F$ is locally bounded. Suppose that 
    \begin{enumerate}
        \item[(A1)] there exists $\cM\subset \RR^{2\times2}$  ($\varnothing\neq\cM$) such that $(v,F,M)$ is a finite-dimensionally-driven self-synchronous intertwinement {of dimension $N_{ss}$} for all $M\in \cM$;
        \item[(A2)] there exists $\cUt\subset\cM$ ($\varnothing\neq\cUt$) such that $(v,F,M)$ is a continuous data assimilation algorithm at scale $N_0$ of strong-type {for all $M\in\cUt$};
        \item[(A3)] there exists $m\in\RR$ such that $\diag(m,m)\in\cM$ and $M_{0\pm}\in\cUt$ whose non-zero entries are equal to $m$, i.e., $M_{0+}$$=${\tiny{$\begin{pmatrix}m & m\\ 0 & 0\end{pmatrix}$}} and $M_{0-}$$=${\tiny{$\begin{pmatrix}0 & 0\\ m & m\end{pmatrix}$}}.
    \end{enumerate}
If every continuous data assimilation algorithm prescribed as above possesses the reconstruction property at scale $N_0$, then the Navier-Stokes process has the finite determining modes property. Moreover, the number of determining modes is bounded above by $\max\{N_{ss}, N_{0}\}$.
\end{Thm}

\begin{proof}
Let $N_*:=\max\{N_{ss}, N_0\}$ and fix {an arbitrary} $N\geq N_*$. For $j=1,2$, let $u_j(t)=u_j(t;u_0^j)$ satisfy \eqref{eq:nse:coupled}, where $u_0^j\in V$ is arbitrary. Suppose that 
    \begin{align}\label{eq:dm:hypothesis}
        \lim_{t\goesto\infty}|P_Nu_1(t)-P_Nu_2(t)|=\lim_{t\goesto\infty}|f_1(t)-f_2(t)|=0.
    \end{align}
We claim that $\lim_{t\goesto\infty}|u_1(t)-u_2(t)|=0$.

Given any $v_0^j\in V$, let $V_+(t)=(v_1(t;v_0^1,f_1), v_2(t;u_0^1,f_1))$ and observe that $V_+$ satisfies \eqref{eq:nse:intertwined} with $g_1\equiv g_2:=f_1$ and any $M_+\in\cUt$. Let $\pi_j$ denote the coordinate projection onto component $j$. Observe that $\pi_2V_+=u_1$. Similarly, let $V_-(t)=(v_1(t;u_0^2,f_2), v_2(t;v_0^2,f_2))$. Then $V_-$ satisfies \eqref{eq:nse:intertwined} with $g_1\equiv g_2:=f_2$ and any $M_-\in\cUt$, so that $\pi_1V_-=u_2$. Then
    \begin{align}\label{eq:u1u2}
        |u_1(t)-u_2(t)|&\leq |u_1(t)-\pi_1V_+(t)|+|\pi_1V_+(t)-\pi_2V_-(t)|+|\pi_2V_-(t)-u_2(t)|\notag\\
        &=|\pi_1V_+(t)-\pi_2V_+(t)|+|\pi_1V_+(t)-\pi_2V_-(t)|+|\pi_1V_-(t)-\pi_2V_-(t)|.
    \end{align}
In particular, $(V_\pm, F,M_\pm)$ are continuous data assimilation algorithms at scale $N_0$. Since $N\geq N_0$, by the assumed reconstruction property, it follows that
    \begin{align}\label{eq:sync:property}
        \lim_{t\goesto\infty}|\pi_1V_+(t)-\pi_2V_+(t)|=\lim_{t\goesto\infty}|\pi_1V_-(t)-\pi_2V_-(t)|=0.
    \end{align}
It therefore remains to show that 
    \begin{align}\label{eq:third:term}
        \lim_{t\goesto\infty}|\pi_1V_+(t)-\pi_2V_-(t)|=0.
    \end{align}

To this end, observe that by $(A3)$, we may choose $M_{0\pm}\in\cUt$, where $m=(M_{0+})_{11}=(M_{0+})_{12}=(M_{0-})_{21}=(M_{0-})_{22}$, such that $M_{\pm}=M_{0\pm}$. It follows that the variable $V:=(\pi_1V_+,\pi_2V_-)=(V_1,V_2)$ 
obeys the following intertwinement:
        \begin{align}\label{eq:intertwinement:M0}
         \begin{split}
        \frac{dV_1}{dt}+\nu AV_1+B(V_1,V_1)&=g_1+mF(V_1),\quad V_1(0)=v_0^1
        \\
         \frac{dV_2}{dt}+\nu AV_2+B(V_2,V_2)&=g_2+mF(V_2),\quad V_2(0)=v_0^2,
         \end{split}
    \end{align}
where $g_1:= f_1+ mF(u_1)$ and $g_2:=f_2+ mF(u_2)$. Thus \eqref{eq:intertwinement:M0} is an intertwinement with intertwining matrix $\diag(m,m)\in\cM$. By $(A1)$, we know that \eqref{eq:intertwinement:M0} is finite-dimensionally-driven self-synchronous. To verify that \eqref{eq:third:term} holds, it thus suffices to verify that
    \begin{align}\label{eq:fdss:verify}
        \lim_{t\goesto\infty}|P_NV_1(t)-P_NV_2(t)|=\lim_{t\goesto\infty}|g_1(t)-g_2(t)|=0.
    \end{align}

The second limit in \eqref{eq:fdss:verify} is zero since $(u_1,u_2)$ is a Foias-Prodi pair at scale $N\geq N_*\geq N_0$ and $F$ satisfies \eqref{eq:F:adapted}. Indeed, we have
    \begin{align}
        |g_1(t)-g_2(t)|\leq |f_1(t)-f_2(t)|+|m||F(u_1(t))-F(u_2(t))|.\notag    
    \end{align}
Since \eqref{eq:dm:hypothesis} holds, it follows that $\lim_{t\goesto\infty}|g_1(t)-g_2(t)|=0$. 

To see that the first limit in \eqref{eq:fdss:verify} is zero, observe that
    \begin{align}\notag
        |P_NV_1(t)-P_NV_2(t)|&\leq |P_NV_1(t)-u_1(t)|+|P_Nu_1(t)-P_Nu_2(t)|+|P_Nu_2(t)-P_NV_2(t)|\notag
        \\
        &\leq |V_1(t)-u_1(t)|+|P_Nu_1(t)-P_Nu_2(t)|+|u_2(t)-V_2(t)|\notag
        \\
        &=|\pi_1V_+(t)-\pi_2V_+(t)|+|P_Nu_1(t)-P_Nu_2(t)|+|\pi_1V_-(t)-\pi_2V_-(t)|.\notag
    \end{align}
By \eqref{eq:dm:hypothesis} and \eqref{eq:sync:property}, it follows that
    \begin{align}\notag
        \lim_{t\goesto\infty} |P_NV_1(t)-P_NV_2(t)|=0.
    \end{align}
We therefore deduce that \eqref{eq:fdss:verify} holds, which implies that \eqref{eq:third:term} holds, and ultimately, from \eqref{eq:u1u2}, that $|u_1(t)-u_2(t)|\goesto0$ as $t\goesto\infty$, holds.

Lastly, since $N\geq N_*$ was arbitrary, we observe that the number of determining modes is bounded above by $N_*$.
\end{proof}

\subsection{Examples of Intertwinements}\label{sect:examples}

In what follows, we identify several non-trivial choices of $(F,M)$ for which \eqref{eq:nse:intertwined} is {finite-dimensionally} self-synchronous. Furthermore, we identify two particular types of intertwinements for which the continuous data assimilation algorithms previously studied in \cite{OlsonTiti2003} and \cite{AzouaniOlsonTiti2014} can be realized as special cases.

\subsubsection{Nudging Intertwinement}\label{sect:intertwinement:nudge}

{
\begin{Def}\label{def:intertwined:nudge}
Given ${g_1}, {g_2}\in L^\infty(0,\infty;H)$, $N>0$, and matrix $M\in\RR^{2\times 2}$, consider:
    \begin{align}\label{eq:intertwined:nudge}
        \begin{split}
        \frac{dv_1}{dt} + \nu Av_1 + B(v_1,v_1) 
        &= 
        {g_1}+ m_{11}P_N v_1+ m_{12}P_N v_2,
	\\
        \frac{dv_2}{dt} + \nu Av_2 + B(v_2,v_2) 
        &= 
        {g_2} +m_{21}P_N v_1+ m_{22}P_N v_2.
        \end{split}
    \end{align}
We refer to \eqref{eq:intertwined:nudge} as the \textbf{nudging intertwinement} whenever $M$ belongs to either of the classes
    \begin{align}\label{def:intertwined:nudge:matrix}
    \cM_{\mu}^{sym}:=\left\{\begin{pmatrix}-\mu_1&\mu_2\\\mu_2&-\mu_1\end{pmatrix}:\mu_1,\mu_2\geq0\right\},\quad \cM_{\mu}^{mut}:=\left\{ \begin{pmatrix}-\mu_1&\mu_1\\ \mu_2&-\mu_2\end{pmatrix}:\mu_1,\mu_2\geq0\right\}.
    \end{align}
When $M$ is given by the first class of matrices, we refer to the intertwinement as the \textbf{symmetric nudging intertwinement}, while the intertwinements corresponding to the second class of matrices will be referred to as the \textbf{mutual nudging intertwinement}.
\end{Def}

In demonstrating both the global well-posedness issue and the finite-dimensional self-synchronous property for the nudging intertwinement, the key idea is to make use of the dissipative properties of the intertwining functions; the structure of the matrices in \eqref{def:intertwined:nudge:matrix} are suitable for this purpose. We are then ultimately able to prove the following results.

\begin{Thm}\label{thm:intertwined:nudge:gwp}
The initial value problem corresponding to the nudging intertwinement is globally well-posed.
\end{Thm}

As in \cref{sect:intertwinement:sync}, we show that the nudging intertwinement is also contains data assimilation algorithms.

\begin{Prop}\label{prop:nudge:strong}
The nudging intertwinement is a continuous data assimilation algorithm of strong-type at scale $K$ whenever $M\in\cM_\mu^{mut}$ and $\mu_1=0$ or $\mu_2=0$
\end{Prop}

\begin{proof}
Firstly, $M_{\pm}\in\cM_\mu^{mut}$ such that $M_-=\begin{pmatrix}0&0\\ \mu_2&-\mu_2\end{pmatrix}$ and $M_+=\begin{pmatrix}-\mu_1&\mu_1\\ 0&0\end{pmatrix}$ clearly belong to $\cU$. Secondly, let $F(u)=P_Ku$. Then $F:H\goesto H$ is uniformly continuous. Moreover, $F=F\circ P_K$. Therefore, by \cref{lem:dm:cda}, it follows that $(F,M)$ defines a continuous data assimilation algorithm of strong-type at scale $K$.
\end{proof}

In the companion work \cite{CarlsonFarhatMartinezVictor2025b}, we will also prove the following.

\begin{Thm}\label{thm:intertwined:nudge:fdss}
The nudging intertwinement is finite-dimensionally-driven self-synchronous.
\end{Thm}

We once again emphasize that an interpretation of \cref{thm:intertwined:nudge:fdss} is that the nudging algorithm for data assimilation satisfies a \textit{finite determining modes-type property}. Under additional assumptions on the intertwining matrix, we are able to show that the nudging intertwinement is self-synchronous.

\begin{Thm}\label{thm:intertwined:nudge:ss}
For each $g_1,g_2\in L^\infty(0,\infty;H)$, there exists $K_*>0$ such that the nudging intertwinement is self-synchronous whenever
    \begin{align}\notag
        K\geq K_*\quad\text{and}\quad \mu_1+\mu_2\geq\frac{1}4K_*^2.
    \end{align}
\end{Thm}

We immediately recover the result of \cite{AzouaniOlsonTiti2014} which asserts that the nudging algorithm for continuous data assimilation solves the state reconstruction problem.

\begin{Cor}\label{cor:AOT}
Given $f\in L^\infty(0,\infty;H)$, let $u$ denote the unique strong solution of \eqref{eq:nse:ff} and $\tv$ denote the corresponding nudged variable, so that $(u,\tv)$ satisfies \eqref{eq:nudge:AOT}. There exists $N_*$ sufficiently large such that 
    \[
        \lim_{t\goesto\infty}|\tv(t)-u(t)|=0,\quad \text{whenever}\quad
        N\geq N_*\quad\text{and}\quad \mu_1+\mu_2\geq\frac{1}4N_*^2.
    \]
\end{Cor}

\begin{proof}
We apply \cref{thm:intertwined:nudge:ss} with intertwining matrix $M\in\cM_\mu^{mut}$ such that $\mu_1=0$ and $\mu_2=\mu$, so that $(v_1,v_2)=(u,\tv)$. Then we set $N_*=K_*$, where $K_*$ is determined by \cref{thm:intertwined:nudge:ss}.
\end{proof}

}

On the other hand, the reconstruction property of the nudging algorithm was proven by \cite{AzouaniOlsonTiti2014}. As a corollary of this result, \cref{prop:nudge:strong}, and \cref{thm:intertwined:nudge:fdss}, we automatically deduce the following corollary of \cref{thm:main}:

\begin{Cor}\label{cor:intertwined:nudge:rpdmp}
The reconstruction property of the nudging algorithm implies that the Navier-Stokes process has the finite determining modes property.
\end{Cor}

In particular, we recover the results \cite[Theorem 3.10, 3.12]{BiswasHudsonLariosPei17} in the context of determining modes.

\subsubsection{Direct-Replacement Intertwinement}\label{sect:intertwinement:sync}

\begin{Def}\label{def:intertwined:sync}
Given ${g_1}, {g_2}\in L^\infty(0,\infty;H)$, $K>0$, and matrix $M\in\RR^{2\times 2}$, consider:
	\begin{align}\label{eq:intertwined:sync}
		\begin{split}
		\bdy_tv_1+\nu Av_1+B(v_1,v_1)&={g_1}+m_{11}P_KB(v_1,v_1)+m_{12}P_KB(v_2,v_2)
		\\
		\bdy_tv_2+\nu Av_2+B(v_2,v_2)&={g_2}+m_{12}P_KB(v_1,v_1)+m_{22}P_KB(v_2,v_2).
		\end{split}
	\end{align}
We refer to \eqref{eq:intertwined:sync} as the \textbf{direct-replacement intertwinement} whenever $M$ belongs to the class
    \begin{align}\label{def:intertwined:sync:matrix}
        \begin{split}
        \cM_{\tht}^{{sym}}&:=\left\{\begin{pmatrix}\tht_1&-\tht_2\\-\tht_2&\tht_1\end{pmatrix}: \tht_1+\tht_2=1\right\}
        \\
        \cM_{\tht}^{{mut}}&:=\left\{\begin{pmatrix}\tht_1&-\tht_1\\ -\tht_2&\tht_2\end{pmatrix}: \tht_1+\tht_2=1,\ \tht_1,\tht_2\geq0\right\}.
        \end{split}
    \end{align}
When $M\in{\cM_{\tht}^{sym}}$, we refer to the \eqref{eq:intertwined:sync} as the \textbf{symmetric direct-replacement intertwinement}, while the intertwinement corresponding to the second class of matrices are referred to as the \textbf{mutual direct-replacement intertwinement}.
\end{Def}

{In a follow-up work \cite{CarlsonFarhatMartinezVictor2025b}, we show that {these intertwinements are well-defined} in the sense of \cref{def:gwp}, and are finite-dimensionally-driven self-synchronous in the sense of \cref{def:fdss}. Nevertheless, a few relevant remarks are in order. Note that when $M=\diag(1,1)$, i.e., the identity matrix, the system \eqref{eq:intertwined:sync} is decoupled. In particular, the assertion that \eqref{eq:intertwined:sync} is finite-dimensionally-driven self-synchronous is {equivalent to the statement that a particular perturbed Navier-Stokes equation possesses the determining modes property.} On the other hand, when $\tht_1=0$ or $\tht_2=0$, then \eqref{eq:intertwined:sync} reduces to \eqref{eq:sync:OT}.
In all cases}, the main mathematical difficulty that must be dealt with in regards to the well-posedness of {the initial value problem corresponding to} \eqref{eq:intertwined:sync} is the loss of energy and enstrophy conservation (when $\nu=0$) due to the truncation of the quadratic nonlinearity. In our follow-up work \cite{CarlsonFarhatMartinezVictor2025b}, we develop an apriori analysis that overcomes this apparent obstruction by observing that the corresponding low-mode evolution is governed by a heat equation. This will ultimately allow us to prove the following.

\begin{Thm}\label{thm:sync:gwp}
The initial value problem corresponding to the direct-replacement intertwinement is globally well-posed for all $M\in\cM_\tht^{mut}$ and also for $M\in\cM_\tht^{sym}$ such that $\tht_1\in\{1/2,1\}$.
\end{Thm}

As \cref{thm:sync:gwp} suggests, the well-posedness of the symmetric direct-replacement intertwinement is more nuanced than its mutual counterpart. In particular, we are able to prove the following.

\begin{Thm}\label{thm:sync:sym:gwp}
For each $g_1,g_2\in L^\infty(0,\infty;H)$ and $v_0^1,v_0^2\in V$, there exists $K_1^*, K_2^*>0$ and $\de_1^*,\de_2^*>0$ such that the corresponding initial value problem associated to the symmetric direct-replacement intertwinement is globally well-posed, for all $K\geq K_1^*$ and  $\tht_1\geq1-\de_1^*$ or $K\geq K_2^*$ and  $|\tht_1-\tht_2|\leq \de_2^*$.
\end{Thm}

{
Both \cref{thm:sync:gwp} and \cref{thm:sync:intertwinement} are proven in the follow-up work \cite{CarlsonFarhatMartinezVictor2025b}. Nevertheless, let us verify that \eqref{eq:intertwined:sync} encompasses continuous data assimilation algorithms in the sense of \cref{def:cda}.

\begin{Prop}\label{thm:sync:cda}
The direct-replacement intertwinement is a continuous data assimilation algorithm of weak-type for $M\in\cM_\tht$ such that $\tht_1=0$ or $\tht_2=0$.
\end{Prop}

\begin{proof}
It suffices to consider the case $\tht_1=0$. Firstly, observe that $M=\begin{pmatrix}0&0\\ 1&1\end{pmatrix}\in\cU$. Secondly, given $u,v\in V$, let $w=u-v$. Then for $F(u)=P_KB(u,u)$, we have
    \begin{align}\notag
        F(u)-F(v)=P_KB(w,w)+P_KB(v,w)+P_KB(w,v).
    \end{align}
By H\"older's inequality, \eqref{est:Bernstein}, and \eqref{est:interpolation}, we have
    \begin{align}\notag
        |F(u)-F(v)|&\leq Kc_L\left(\|w\||w|^{1/2}+2\|v\|^{1/2}|v|^{1/2}\|w\|^{1/2}\right)|w|^{1/2}.
    \end{align} 
Thus, $F$ satisfies \eqref{cond:F:cts}. An application of \cref{lem:dm:cda} concludes the proof.
\end{proof}

We conclude this section with the main results regarding the self-synchronous property of the direct-replacement intertwinement.
}

\begin{Thm}\label{thm:sync:intertwinement}
The direct-replacement intertwinement is self-synchronous for all $M\in\cM_\tht^{mut}$ and also for $M\in\cM_\tht^{sym}$ such that $\tht_1\in\{1/2,1\}$.
\end{Thm}

Due to the more nuanced well-posedness theory for the symmetric direct-replacement intertwinement, one has a correspondingly nuanced statement for it when it is self-synchronous.

\begin{Thm}\label{thm:sync:intertwinement:sym}
For each $g_1,g_2\in L^\infty(0,\infty;H)$ and $v_0^1,v_0^2\in V$, there exists $K_1^*, K_2^*>0$ and $\de_1^*,\de_2^*>0$ such that corresponding symmetric direct-replacement intertwinement is self-synchronous whenever $K\geq K_1^*$ and $\tht_1\geq 1-\de_1^*$ or $K\geq K_2^*$ and $|\tht_1-\tht_2|\leq \de_2^*$.
\end{Thm}

Analogously to \cref{cor:AOT}, we immediately recover the result of \cite{OlsonTiti2003} regarding the reconstruction property of the direct-replacement algorithm for CDA.

\begin{Cor}\label{cor:OT}
Given $f\in L^\infty(0,\infty;H)$, let $u$ denote the unique strong solution of \eqref{eq:nse:ff} and $v$ denote the corresponding assimilated variable, so that $(u,v)$ satisfy \eqref{eq:sync:OT}. There exists $N_*$ sufficiently large such that
    \[
        \lim_{t\goesto\infty}|v(t)-u(t)|=0,
    \]
whenever $N\geq N_*$.
\end{Cor}

\begin{proof}
We apply \cref{thm:sync:intertwinement} with intertwining matrix $M\in\cM_\tht^{mut}$ such that $\tht_1=0$ and $\tht_2=1$, so that $(v_1,v_2)=(u,v)$. Then we set $N_*=K_*$, where $K_*$ is determined by \cref{thm:sync:intertwinement}.
\end{proof}

{

\begin{Rmk}Since the direct-replacement intertwinement only induces continuous data assimilation algorithms of weak-type, \cref{thm:main} does not apply. 
Our framework thus identifies a possible obstruction for the reconstruction property in continuous data assimilation algorithms (in the sense of \cref{def:cda}) to imply the existence of finite determining modes for the Navier-Stokes process. Indeed, in the case of the direct-replacement algorithm, synchronization of the low-modes is forced to occur in a specific way, that is, according to an unforced heat equation (see \cite[Section 3.0.2, Equation 3.8]{CarlsonFarhatMartinezVictor2025b}). In other words, nonlinear effects are not allowed to drive low-mode synchronization error in the direct-replacement algorithm; the intertwining function is chosen to exactly cancel these effects. On the other hand, the determining modes property is formulated (in \cref{def:dm}) in a manner that is indifferent to how the low-mode error is driven to zero. In this way, the direct-replacement algorithm can be viewed as a particular ``instantiation" of the determining modes property. This observation was already known to Olson and Titi in \cite{OlsonTiti2003} and was in fact the source of inspiration for the direct-replacement algorithm, i.e., to view it as a particular use-case of the determining modes property of the Navier-Stokes process. Our framework formalizes this intuition. Moreover, this intuition is manifested in fact that the direct-replacement intertwinement is finite-dimensionally-driven self-synchronous as a consequence of already being self-synchronous (see \cref{thm:sync:intertwinement}, \cref{thm:sync:intertwinement:sym}).


In contrast to the direct-replacement algorithm, the nudging algorithm allows nonlinear effects to drive low-mode synchronization error; the nudging parameters must then be tuned accordingly to stabilize these effects to achieve synchronization.  This indicates that the perturbation, $F(\phi)=P_K\phi$, is in some sense ``weak" relative to the perturbation, $F(\phi)=P_KB(\phi,\phi)$, defining the direct-replacement algorithm. The fact that the nudging algorithm defines a continuous data assimilation algorithm of strong-type is thus a reflection of the fact that it ``weakly enforces" synchronization of low-modes, thereby giving way to a logically stronger notion. 

Nevertheless, in either the case of the direct-replacement or nudging intertwinements, the fact that they are both finite-dimensionally-driven self-synchronous is a precise interpretation of the statement that ``determining modes implies reconstruction property," as long as one is willing to accept that the concept of being finite-dimensionally-driven self-synchronous is the ``proper way" to formulate the determining modes property.
\end{Rmk}
}

\begin{Rmk}\label{rmk:uniform}
The global well-posedness of the nudging and direct-replacement intertwinements each face a common difficulty: obtaining uniform-in-time apriori bounds with suitable on the intertwining matrix. This technical matter is subtle to observe in the proof of \cref{thm:main}, where this issue comes into play. In particular, in writing \eqref{eq:intertwinement:M0}, we see that the forcing terms, $g_1, g_2$, in the intertwinement presented there, depend on $M$ and $F$ coming from a different intertwinement. However, $M$ depends on $\mu_j$, in the case of the nudging intertwinement, or $\tht_j$, in the case of the direct-replacement intertwinement, while $F$ typically depends on $N$. In the case of the nudging intertwinement, this issue is particularly acute since one must ensure consistency with the choice of $N_*$, which is chosen as the maximum among the scales that ensure synchronization for either the corresponding continuous data assimilation or the intertwinement induced by the proof. Specifically, it is important to show that $N_{ss}$ does not depend strongly on the parameter $m$ in spite of the choice of $g_1,g_2$ depending explicitly on $m$. In \cite{CarlsonFarhatMartinezVictor2025b}, we indeed overcome this issue and in fact show that $N_{ss}$ can be ultimately be chosen independently of $m$. In this case of the direct-replacement intertwinement, this issue can be dealt with by simply using the fact that $\tht_1,\tht_2\in[0,1]$, whereas with the nudging intertwinement, we exploit the sign-definite nature of the intertwining function $F$.

On the other hand, the direct-replacement intertwinement faces more significant difficulties in establishing global-well posedness since the intertwining function, $F$, destroys the energy and enstrophy inequalities typically enjoyed by solutions of the 2D NSE. In their study of direct-replacement algorithm for CDA, the main insight of \cite{OlsonTiti2003} was that the low-mode errors satisfy a heat equation, and therefore decay to zero. In establishing global well-posedness for the direct-replacement intertwinement we also seek exploit this phenomenon, but it is extracted in a much more subtle way. Moreover, the coupling between low-modes and high-modes induced by $M, F$, requires a more sophisticated bootstrap method. These details are borne out in \cite{CarlsonFarhatMartinezVictor2025b}.
\end{Rmk}

\begin{Rmk}\label{rmk:expand}
Note that one may extend the definition of the nudging intertwinement (\cref{def:intertwined:nudge}) to include projections other than the spectral projection, $P_N$, such as the so-called volume element projection or nodal value projection. More generally, one can indeed consider more general operators provided that they satisfy suitable approximation-of-identity properties, as was considered in \cite{AzouaniOlsonTiti2014}. However, we do not consider such {generalizations} here as it is not at the moment clear how to extend the synchronization intertwinement (\cref{def:intertwined:sync}) to accommodate projections other than $P_N$. Since the primary goal of this article is to present a unified theory of intertwinement that contains \textit{both} continuous data assimilation algorithms of \cite{OlsonTiti2003} and \cite{AzouaniOlsonTiti2014}, we therefore defer the study of its generalization to other forms of projection to a future work and simply point out that such a generalization is a very natural and relevant consideration.
\end{Rmk}

\section{Computational Results}\label{sect:numerical}
\noindent In this section, we numerically explore the convergence properties of various the intertwinements introduced in \cref{sect:intertwine} and studied above.

\subsection{Numerical Methods}\label{sect:numerical:tech}

Simulations of the 2D Navier-Stokes equations are performed in MATLAB (R2023b) using a fully dealiased pseudo-spectral code defined on the periodic box $\mathbb{T}^2 = [-\pi,\pi]^2$. That is, the spatial derivatives were calculated by multiplication in Fourier space. The equations were simulated at the stream function level, i.e. the 2D Navier-Stokes equations were written in the following form:
\begin{align}\label{scheme:NSE}
    \begin{split}
\psi_t + \De^{-1}(\nabla^{\perp}\psi\cdot\nabla)\De\psi &= \nu\De\psi + \De^{-1}\nabla^{\perp}\cdot f,
    \end{split}
\end{align}
where $\nabla^\perp=(-\bdy_y,\bdy_x)$ and $\De^{-1}$ denotes the inverse Laplacian, which is taken with respect to the periodic boundary conditions and the mean-free condition. The initial condition and parameters were chosen such that our simulations coincide with a turbulent regime. Specifically, the viscosity, $\nu$ was chosen to be $\nu = 0.0005$, and 
that the body force is as given in \cite{Olson_Titi_2008_TCFD} to be low mode forcing concentrated over a band of frequencies with $10\leq |\vec{k}|^2 \leq 12$. The forcing term is renormalized such that the Grashof number $G = \frac{\norm{f}_{L^\infty}}{\nu^2} = 100,000$. To produce the initial data that we used for our simulations we ran the 2D Navier-Stokes equations forward in time from zero initial data out to time $10,000$. We note that the initial profile is slightly under-resolved as it is slightly above machine precision (approximately $2.2204\times 10^{-16}$) at the 2/3 dealiasing line, see \cref{fig:spectrum}. The spectrum remains well-resolved for the duration of all of our simulations with the exception being the spectrum for certain cases shown in \cref{sect:computations:mutual}.

The time-stepping scheme we utilized was a semi-implicit scheme, where we handle the linear diffusion term implicitly via an integrating factor in Fourier space. For an overview of integrating factor schemes see e.g. \cite{Kassam_Trefethen_2005, Trefethen_2000_spML} and the references contained within. The equations are then evolved using an explicit Euler scheme, with both the nonlinear term and the feedback-control term treated explicitly with the nonlinear term computed using  $2/3$ dealiasing. We used a timestep of $\Delta t = 0.01$. In the following subsections, we present the results of various numerical tests confirming the results of our theorems. That is, we present numerical results indicating that each of the examples of intertwinements exhibit synchronization of $v_1, v_2$, at an exponential rate given that sufficiently many Fourier modes are implemented in the intertwining function.

We emphasize, once again, that the main intent of tests illustrated in this current section are to confirm the theoretical results established in the previous sections. A more comprehensive study probing the dynamical properties of intertwinement in greater generality and its relation to the dynamics of the underlying 2D NSE is most certainly warranted, especially in cases for which rigorous theorems are not currently available or for the cases which do, but are considered \textit{outside} of the parameter regimes asserted by the rigorous theorems. These further investigations will be the primary concern of a future work.

Before we describe the numerical results, we point out to the reader that it is convenient to borrow language from continuous data assimilation and refer to the modes implemented in the intertwining function as the ``observed modes," and the modes complementary to these as the ``unobserved modes."

\subsection{Direct-Replacement Intertwinement}\label{sect:computations:sync}

In this section we test an implementation of the synchronization intertwinement defined in \eqref{eq:intertwined:sync}.
We focus on the \emph{mutual direct-replacement intertwinement} i.e., \eqref{eq:intertwined:sync} with $M$$=${\tiny{$\begin{pmatrix}\tht_1&-\tht_1\\ -\tht_2&\tht_2\end{pmatrix}$}}, where $\tht_1+\tht_2=1$, and $\tht_1,\tht_2\geq0$, and \emph{symmetric direct-replacement intertwinement} i.e., \eqref{eq:intertwined:sync} with $M$$=${\tiny{$\begin{pmatrix}\tht_1&-\tht_2\\ -\tht_2&\tht_1\end{pmatrix}$}}, where $\tht_1+\tht_2=1$. For these intertwinements, we implement the intertwined system according to \eqref{scheme:NSE}, but with additional terms coming from the intertwining functions that are treated explicitly. We simulate these equations for various instances of the intertwining matrix, $M$, using spatial resolution, $N = 2^9$ and viscosity $\nu = 0.0005$. 
We utilized $g_1 = g_2 = f$, with $f$ being the time-independent forcing described in \cref{sect:numerical:tech}. 



\subsubsection{Mutual Direct-Replacement Intertwinement}\label{sect:computations:sync:mutual}

In our computational investigation of the mutual direct-replacement intertwinement we examined the effect of $\theta_1$ on the ability to self-synchronize. The results of these simulations can be seen in \cref{fig:theta trials}. Note that we consider the first $50$ Fourier modes to be used in defining the intertwining function. We also initialize $v_2(t_0) = P_N(v_1(t_0))$. We ultimately observe self-synchronization at an exponential rate in time for any choice of $\theta_1 \in [0,1]$ and that the error dynamics behave qualitatively the same across all values of $\tht_1$. Although we observe some deviation from the typical behavior in the error dynamics of the observed modes in contrast to the unobserved modes, these deviations do not go above $10^{-13}$ during the time simulated.

\begin{figure}

    \centering
    
    \includegraphics[width=0.75\linewidth]{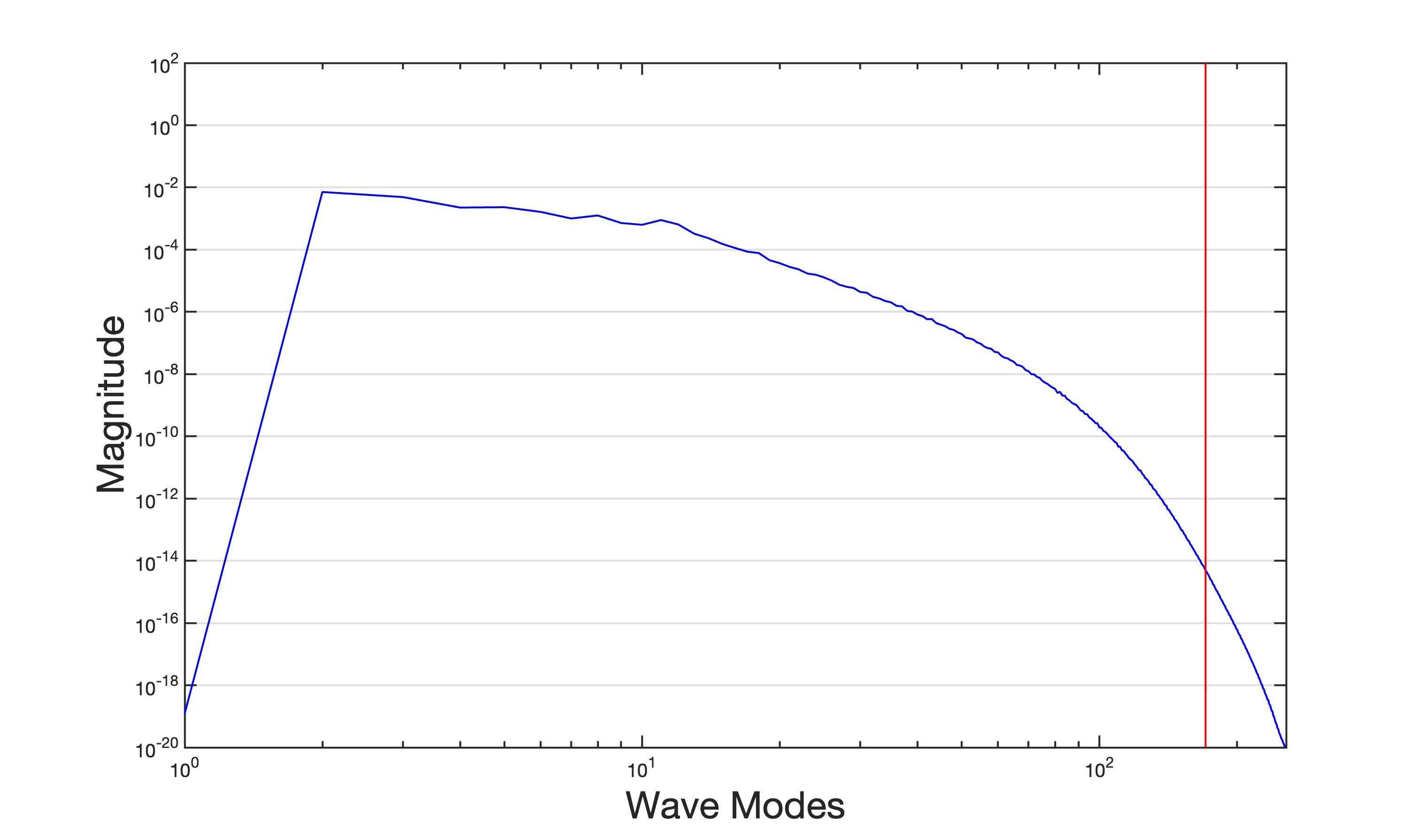}

    \caption{Energy spectrum of the initial data with $\nu = 0.005$, $G = 100,000$, and $\Delta t = 0.001$. The vertical red line is the 2/3 dealiasing cutoff as $\frac{2}{3}\frac{N}{2} = 170.\overline{6}.$}
            \label{fig:spectrum}

\end{figure}

\begin{figure}

\begin{subfigure}[b]{.49\textwidth}
\centering

            \includegraphics[width=\textwidth]{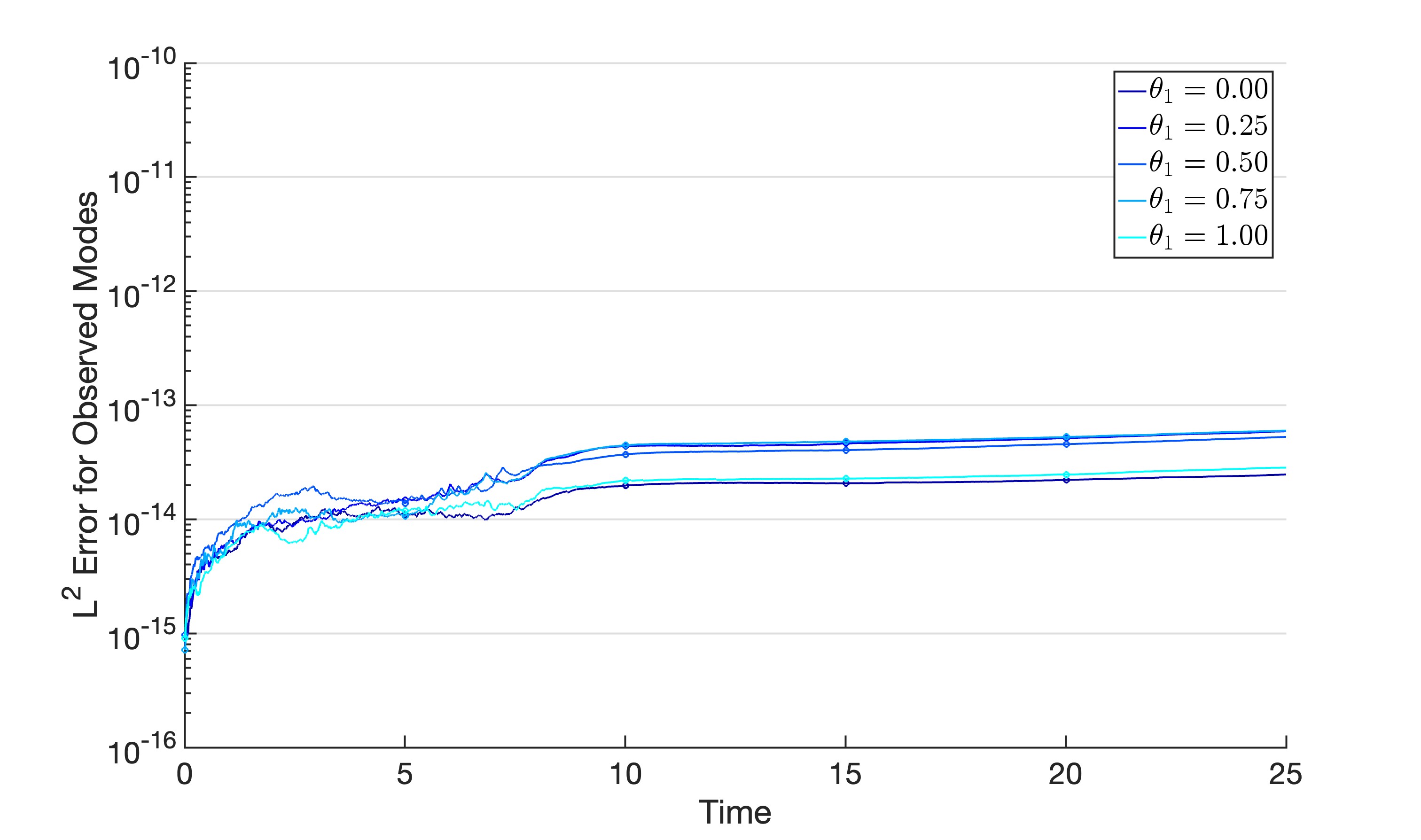}

    \end{subfigure}
    \begin{subfigure}[b]{.49\textwidth}
\centering

            \includegraphics[width=\textwidth]{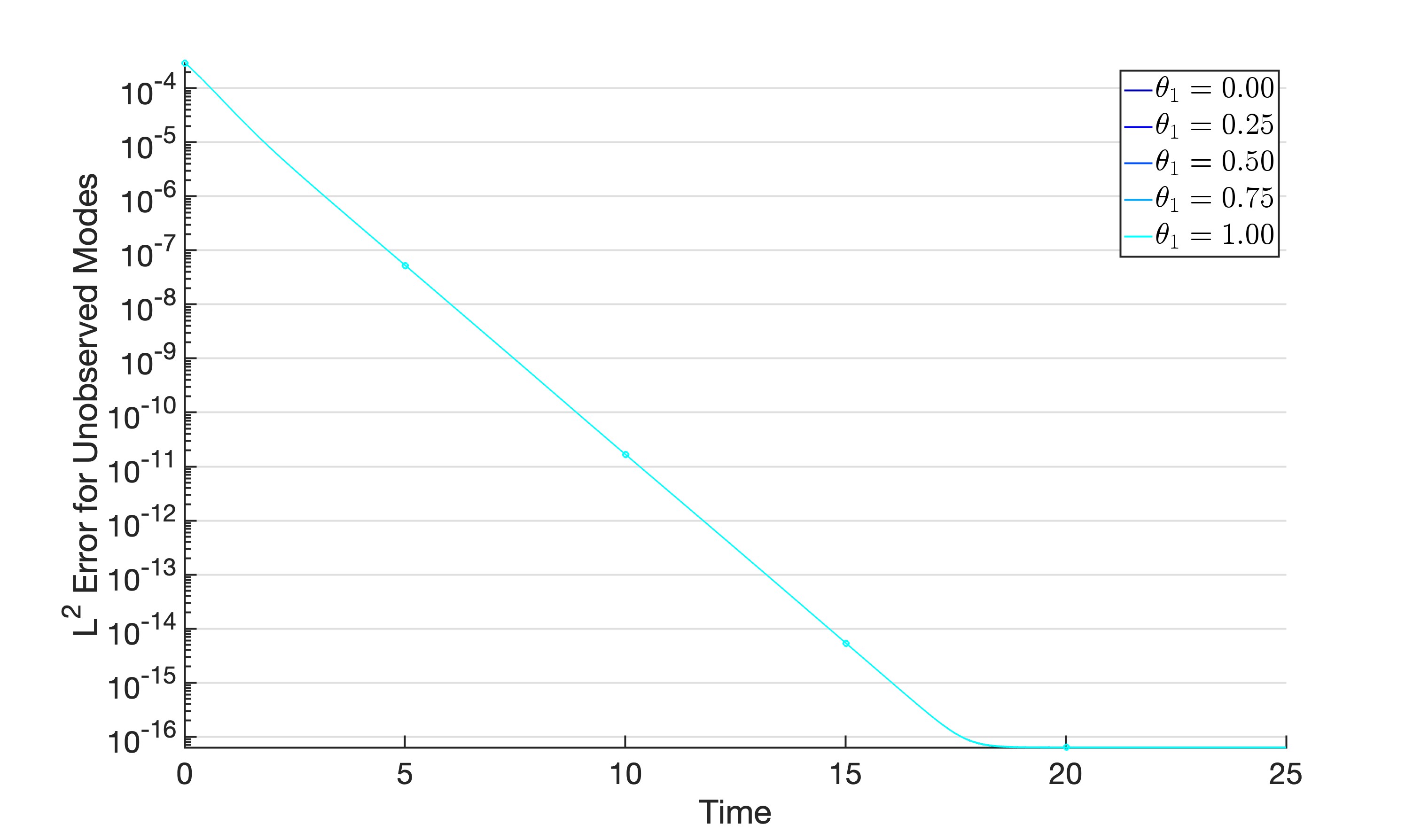}

    \end{subfigure}



    \caption{Error over time for different $\theta_1$ values for mutual direct-replacement. Here $v_2(t_0) = P_N(v_1(t_0))$.}
        \label{fig:theta trials}

\end{figure}

\subsubsection{Symmetric Direct-Replacement Intertwinement} \label{sect:computations:sync:symmetric}
Here we have implemented the equations using our scheme for the 2D NSE given in \eqref{scheme:NSE}, with the additional nonlinear terms computed explicitly. 
We utilized $g_1 = g_2 = f$, with $f$ being the time-independent forcing described in \cref{sect:numerical:tech}.

To investigate the convergence properties of symmetric direct-replacement intertwinements, we considered initial conditions for $v_2$ given as perturbations of $v_0^1$, the intitial condition for $v_1$. In particular, we initialized $v_2$ such that $v_0^2 = v_0^1 + \varepsilon$, where $\varepsilon$ is given in Fourier space as a matrix with random complex coefficients, $\hat{\varepsilon}_{\vec{k}}$, indexed by the wavenumber, $\vec{k} \in \mathbb{Z}^2$. The components of $\hat{\varepsilon}_{\vec{k}}$ are each individually drawn from a Gaussian distribution $\Re \hat{\varepsilon}_{\vec{k}}, \,\Im \hat{\varepsilon}_{\vec{k}}\sim \mathcal{N}(0,\sigma)$, and the global $\varepsilon$ is given by the discrete Fourier transform:
\begin{align}
    \varepsilon(x) = \frac{1}{N}\sum_{\vec{k}\in \mathbb{Z}^2_M} \hat{\varepsilon}_{\vec{k}} e^{i 2\pi \vec{k}\cdot\frac{\vec{x}}{N}},\label{def:noise}
\end{align}
where $\vec{k} \in \mathbb{Z}^2_M = \{\vec{k} \in \mathbb{Z}^2\mid |\vec{k}| \leq M\}$. Here $M$ is the number of observed Fourier modes, which for all simulation present in this section was $M = 50$. As we are considering observations of $\mathbb{R}$-valued solutions, we stipulate that $\hat{\varepsilon}_{-\vec{k}} = \overline{\hat{\varepsilon}_{\vec{k}}}$, where 
$\overline{\,\cdot\,}$ 
denotes complex conjugation. To enforce this, we draw random coefficients for $\hat{\varepsilon}_{\vec{k}}$ for half of the indices of $\vec{k}$ and determine the remaining components to enforce the reality condition. Additionally, we stipulate that $\varepsilon_{\vec{0}} = \vec{0}$, as the systems under consideration are mean-free.

The results of our simulations can be seen in \cref{fig: symmetric-D-R perturbation 1/2,fig: symmetric-D-R IC,fig: symmetric-D-R long time,fig: symmetric-D-R pertubration 1}. We note that while there appears to be some deviation in the low mode error in \cref{fig: symmetric-D-R long time,fig: symmetric-D-R pertubration 1,fig: symmetric-D-R perturbation 1/2}, the error is approximately $10^{-9}$ across all simulations using a consistent $v_0^2$. In the case when the initial condition for $v_2$ was varied for $\theta_1 = \frac{1}{2}$, we see in \cref{fig: symmetric-D-R IC} that the low modes error exhibit slow exponential decay that appears consistent across all trials with the initial condition determining the initial value of the error. 

We found in computations that the error behaved qualitatively similar for all values of $\theta_1,\theta_2 \geq 0$ satisfying $\theta_1 + \theta_2 = 1$. This can be seen in \cref{fig: symmetric-D-R long time}, in which we see that for $\theta_1 \in \{\frac{1}{2}, \frac{3}{4},1\}$ the error in the low modes remains approximately $10^{-9}$ for the simulated times and the high mode error exhibits rapid exponential decay to approximately $10^{-13}$ or below. We note that while the error looks qualitatively similar, ignoring small deviations below $10^{-13}$, the solutions themselves appear qualitatively different for varying $\theta_1, \theta_2$ values. In particular, we note that when $\theta_1 = 1$, the high modes of both $v_1$ and $v_2$ become identically zero as time evolves. This was also seen for all intertwinements that were small perturbations of $\theta_1 = 1$, which can be seen in \cref{fig: symmetric-D-R pertubration 1}.

\begin{figure}

\begin{subfigure}[b]{.49\textwidth}
\centering

            \includegraphics[width=\textwidth]{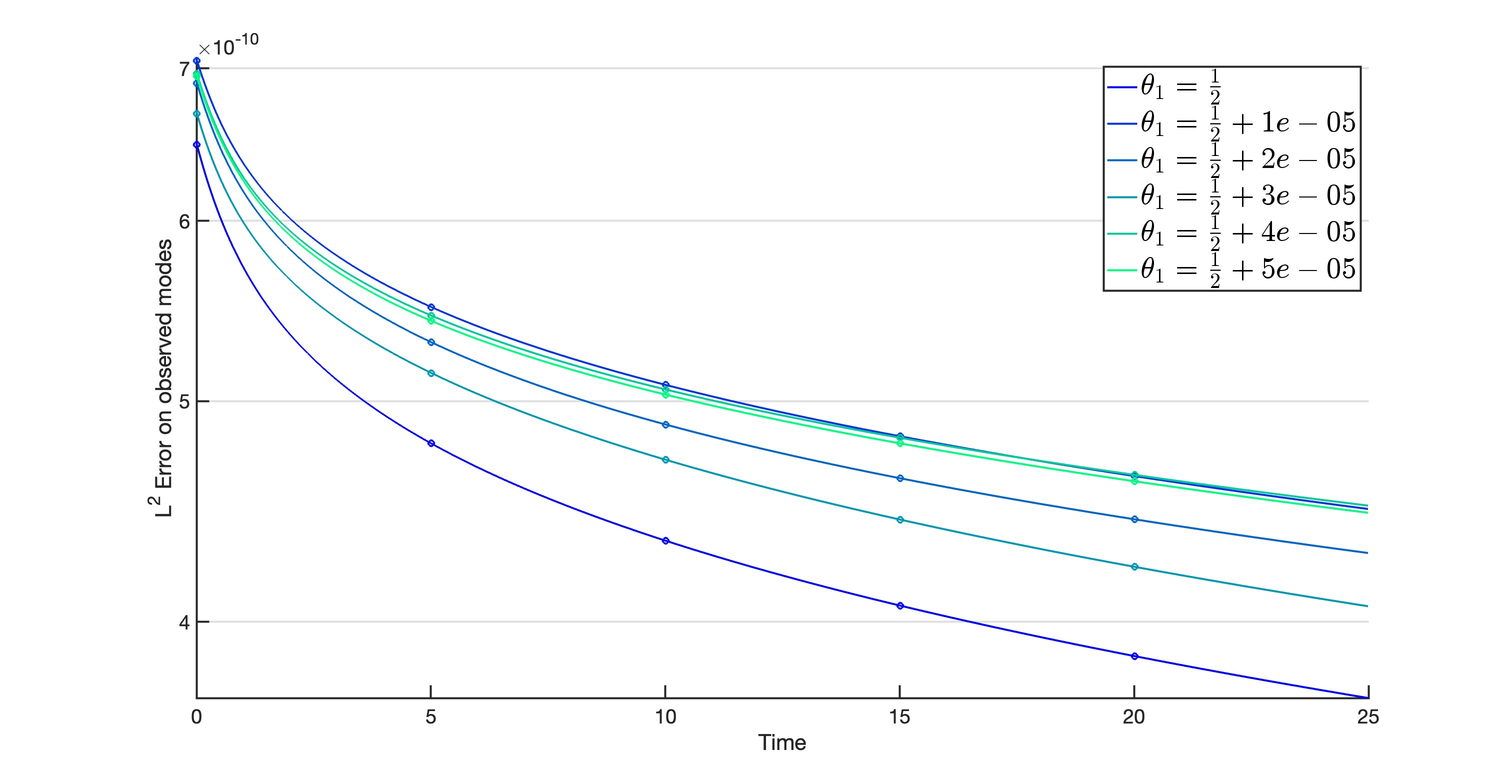}

    \end{subfigure}
    \begin{subfigure}[b]{.49\textwidth}
\centering

            \includegraphics[width=\textwidth]{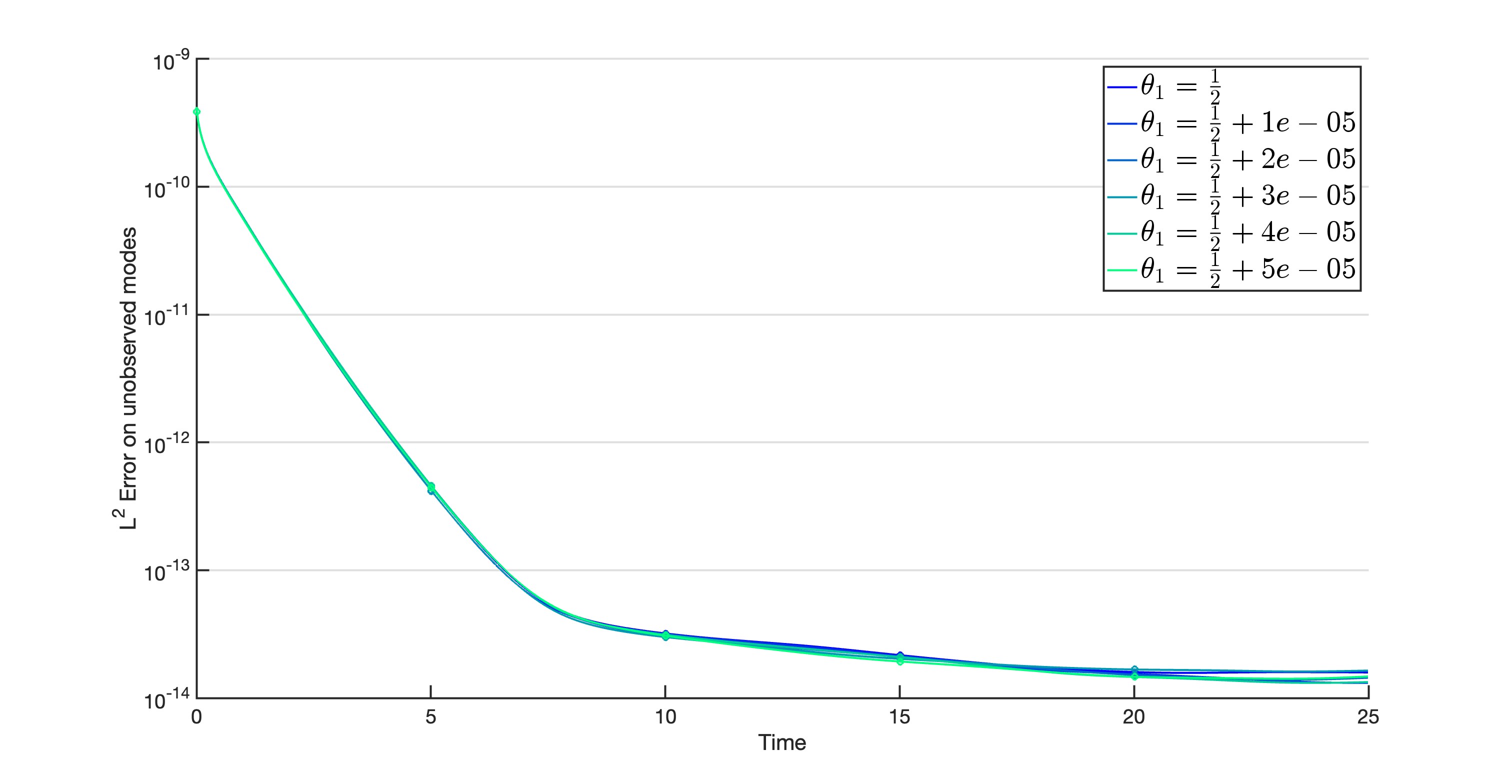}

    \end{subfigure}


    \caption{Error over time for symmetric direct-replacement intertwinements with $\theta_1 = \frac{1}{2} + c$, where $c$ is given in the legend as a multiple of $\Gr^{-1} = 10^{-5}$. Here the initial condition for $v_2$ was given by $v_0^2 = v_0^1 + \varepsilon$, where $\varepsilon$ is given as in \cref{def:noise} with $\sigma^2 = 10^{-10}$.}

    \label{fig: symmetric-D-R perturbation 1/2}
\end{figure}

\begin{figure}

\begin{subfigure}[b]{.49\textwidth}
\centering

            \includegraphics[width=\textwidth]{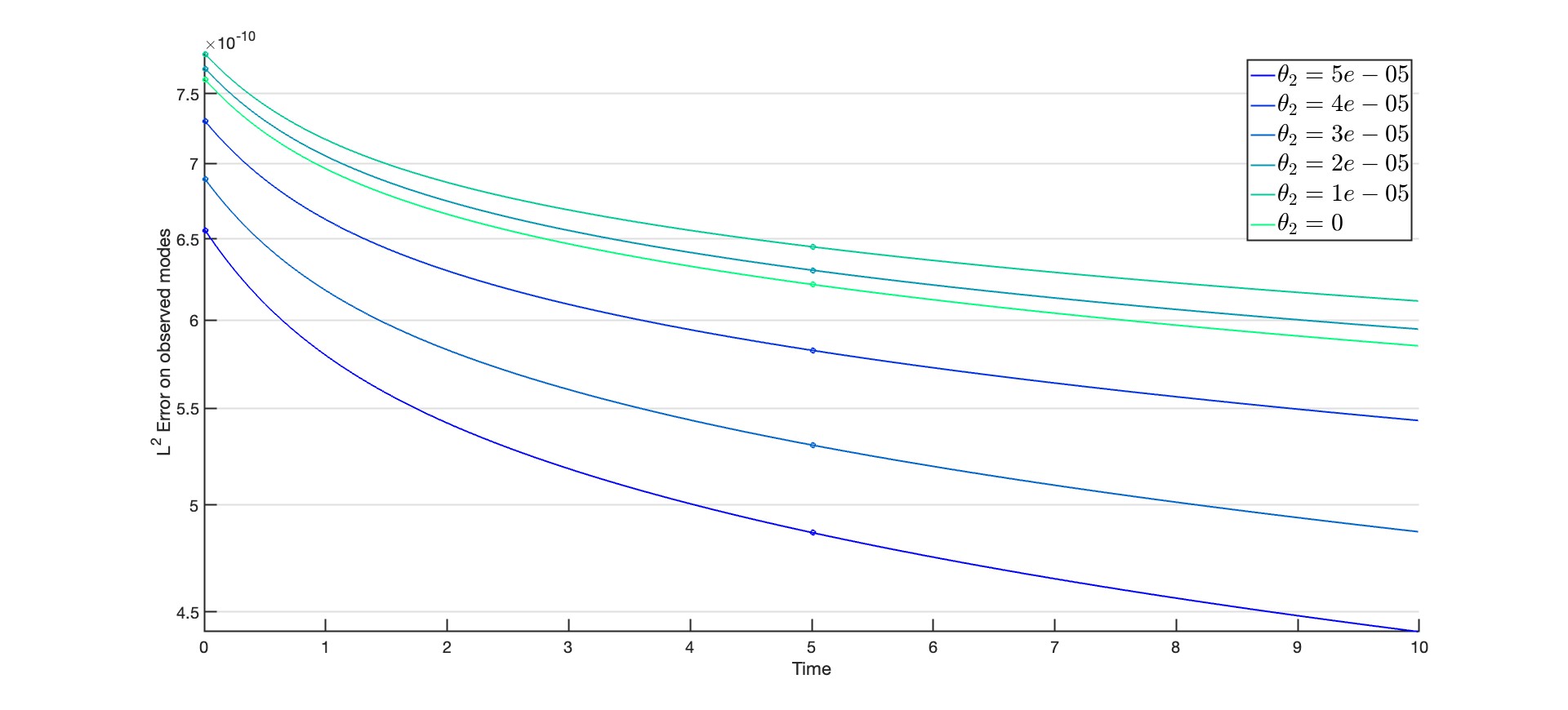}

    \end{subfigure}
    \begin{subfigure}[b]{.49\textwidth}
\centering

            \includegraphics[width=\textwidth]{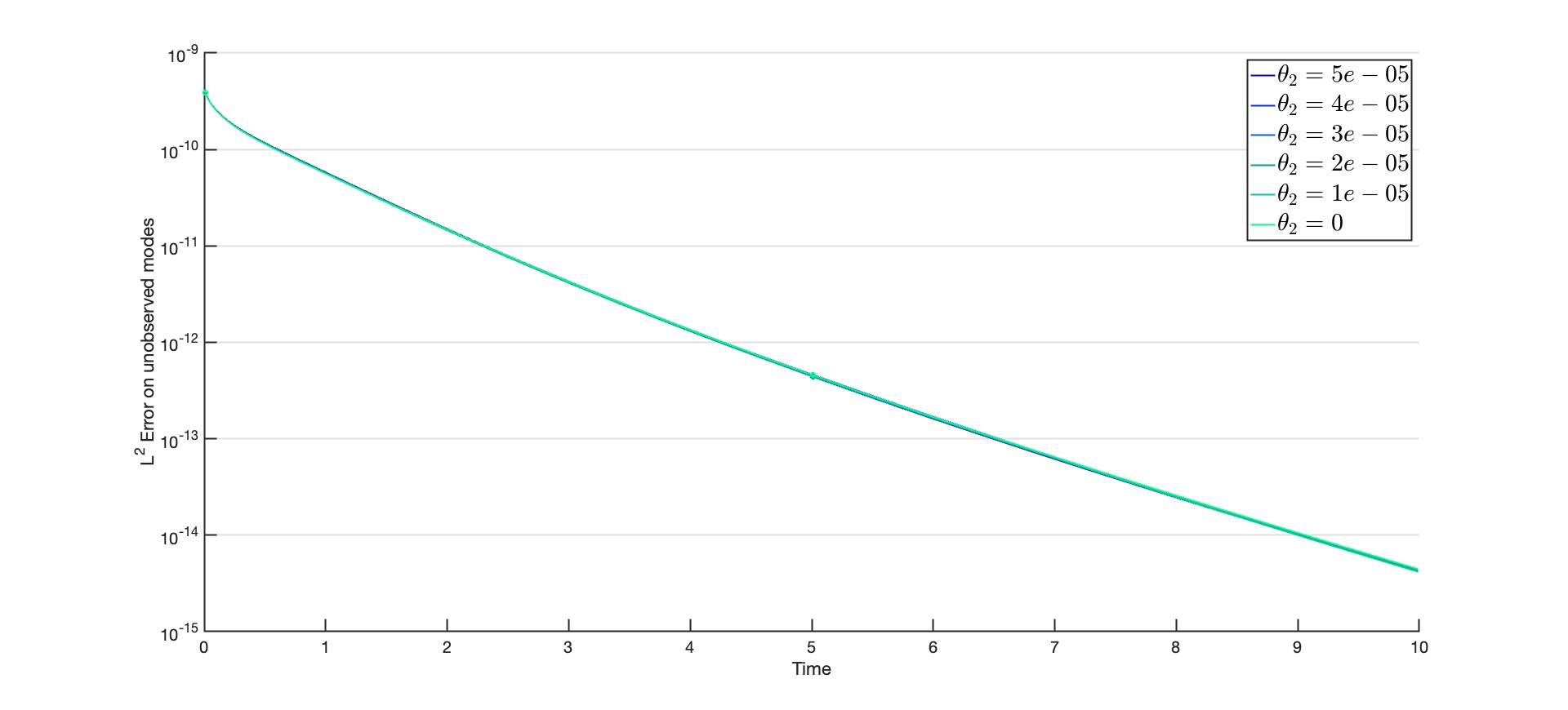}

    \end{subfigure}


    \caption{Error over time for symmetric direct-replacement intertwinements with $\theta_1 = 1 - c$, where $c$ is given in the legend as a multiple of $\Gr^{-1} = 10^{-5}$. Here the initial condition for $v_2$ was given by $v_0^2 = v_0^1 + \varepsilon$, where $\varepsilon$ is given as in \eqref{def:noise} with $\sigma^2 = 10^{-10}$.}

    \label{fig: symmetric-D-R pertubration 1}
\end{figure}

\begin{figure}

\begin{subfigure}[b]{.49\textwidth}
\centering

            \includegraphics[width=\textwidth]{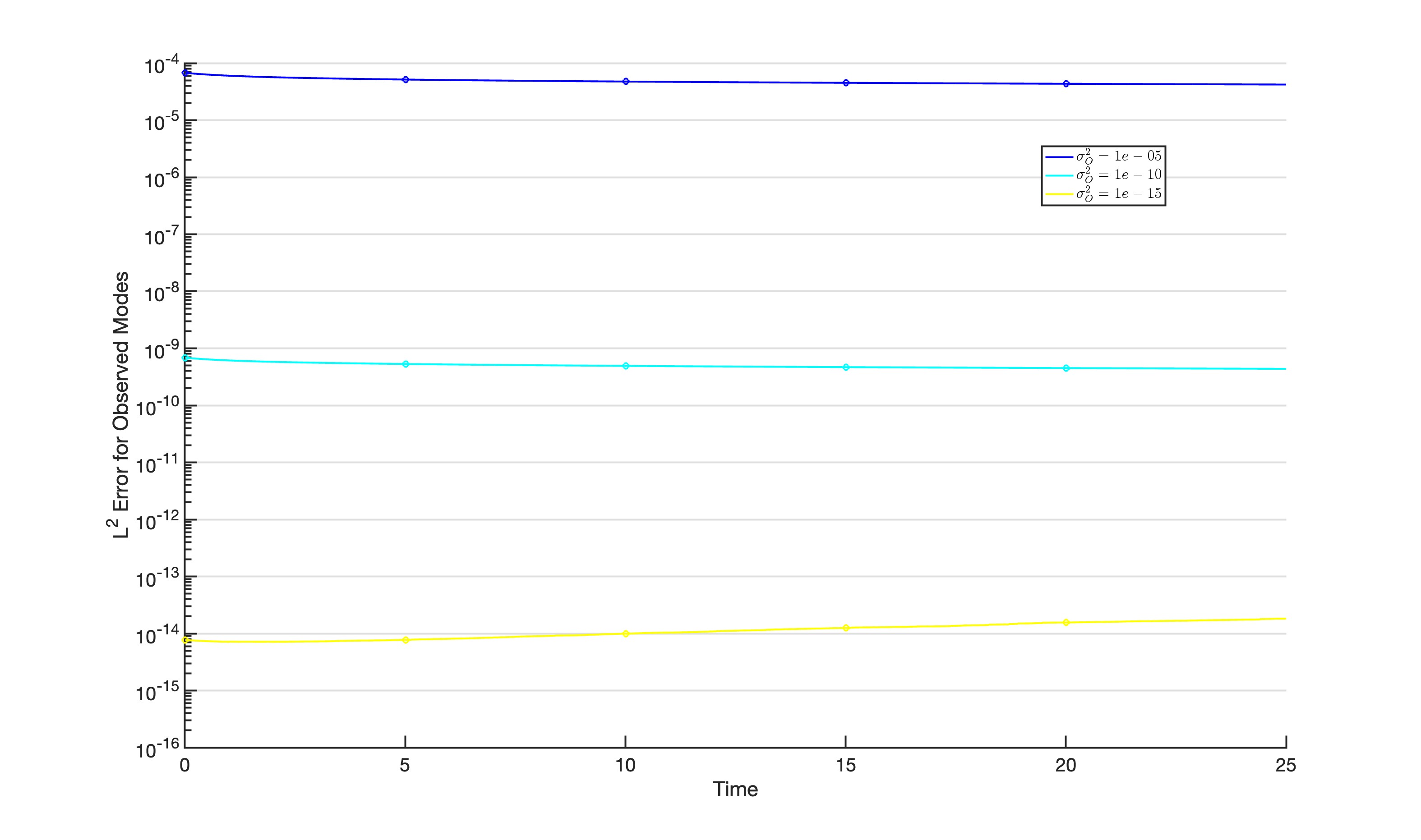}

    \end{subfigure}
    \begin{subfigure}[b]{.49\textwidth}
\centering

            \includegraphics[width=\textwidth]{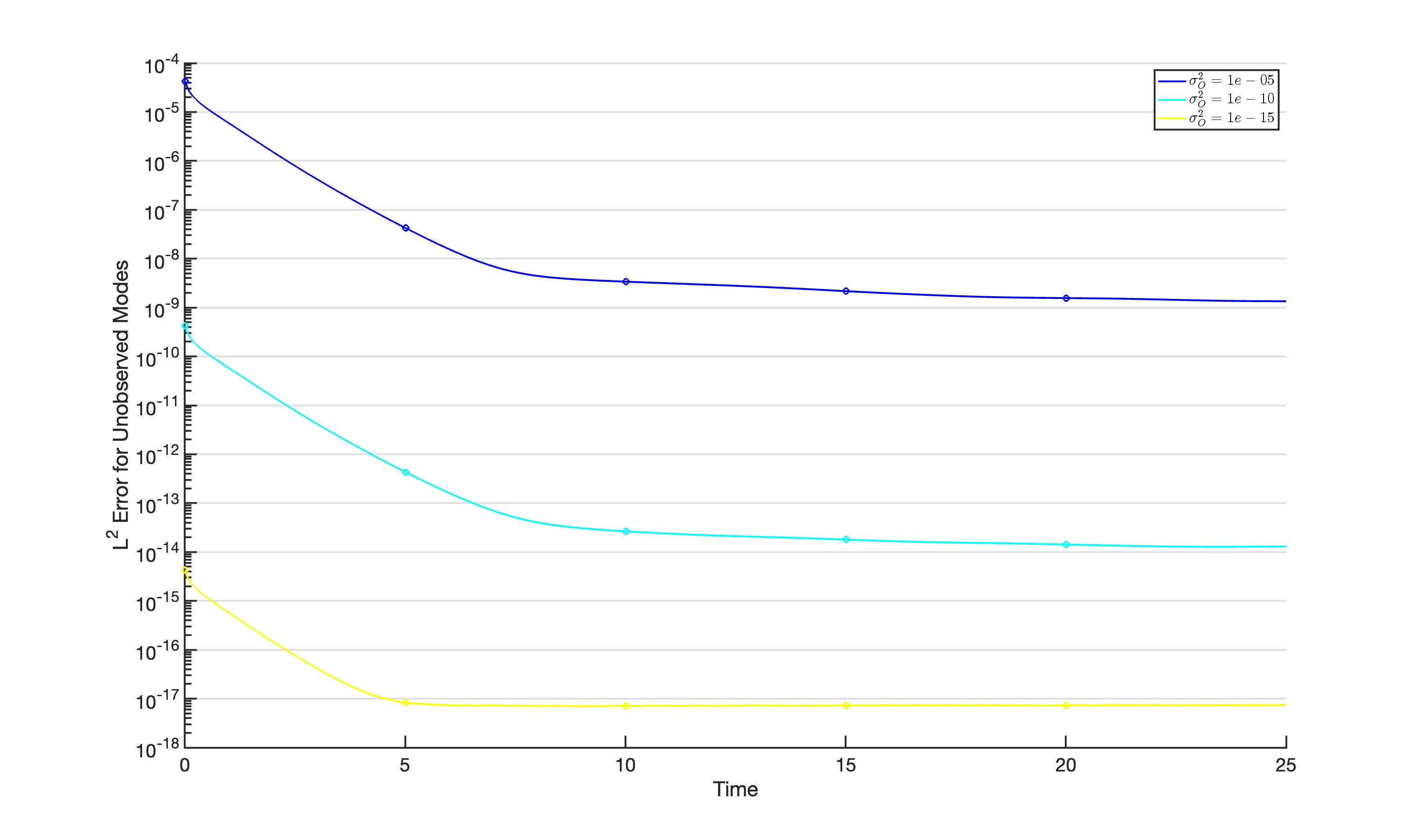}

    \end{subfigure}


    \caption{Error over time for symmetric direct-replacement intertwinements with $\theta_1 = \frac{1}{2}$. Here the initial condition for $v_2$ was given by $v_0^2 = v_0^1 + \varepsilon$, where $\varepsilon$ is given as in \eqref{def:noise} with $\sigma^2$ as in the legend.}

    \label{fig: symmetric-D-R IC}
\end{figure}

\begin{figure}

\begin{subfigure}[b]{.49\textwidth}
\centering

            \includegraphics[width=\textwidth]{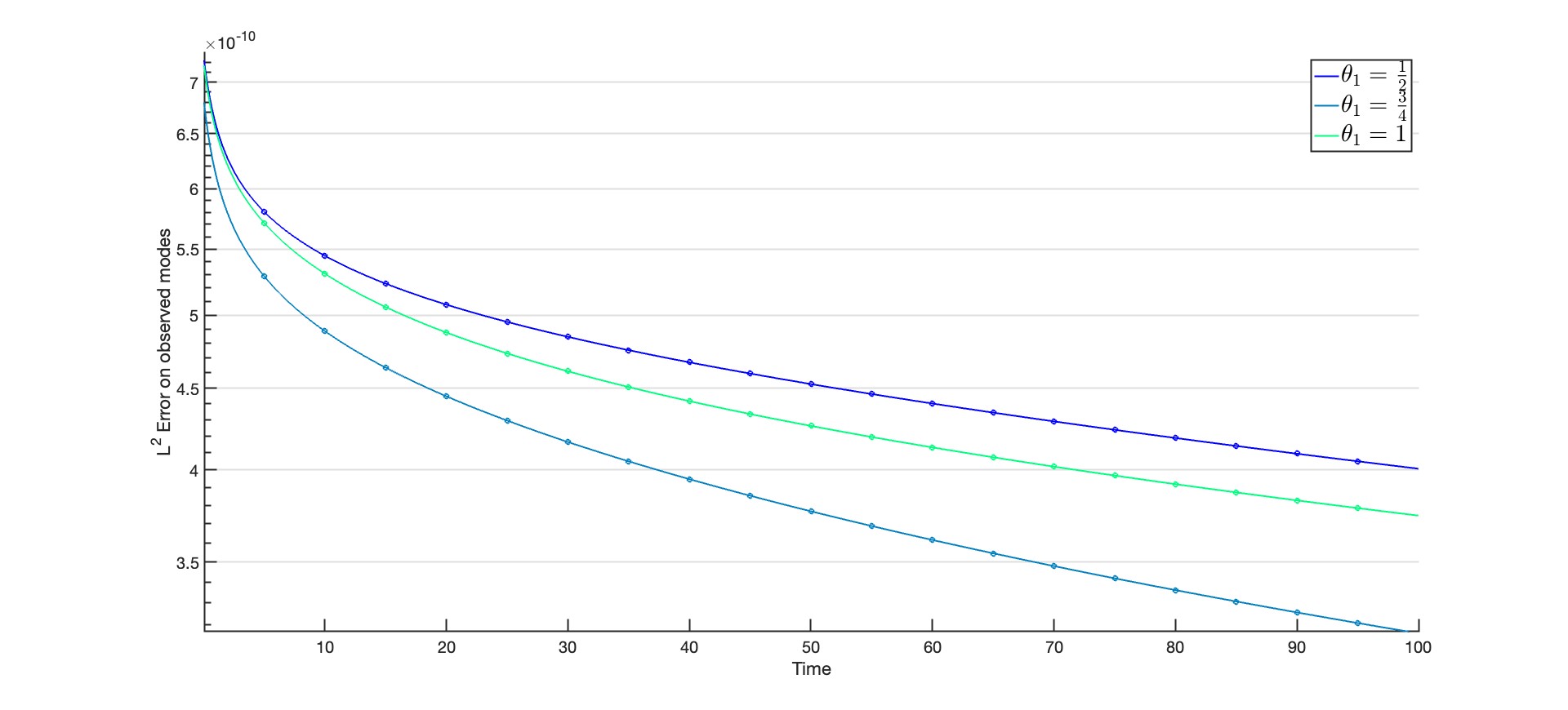}

    \end{subfigure}
    \begin{subfigure}[b]{.49\textwidth}
\centering

            \includegraphics[width=\textwidth]{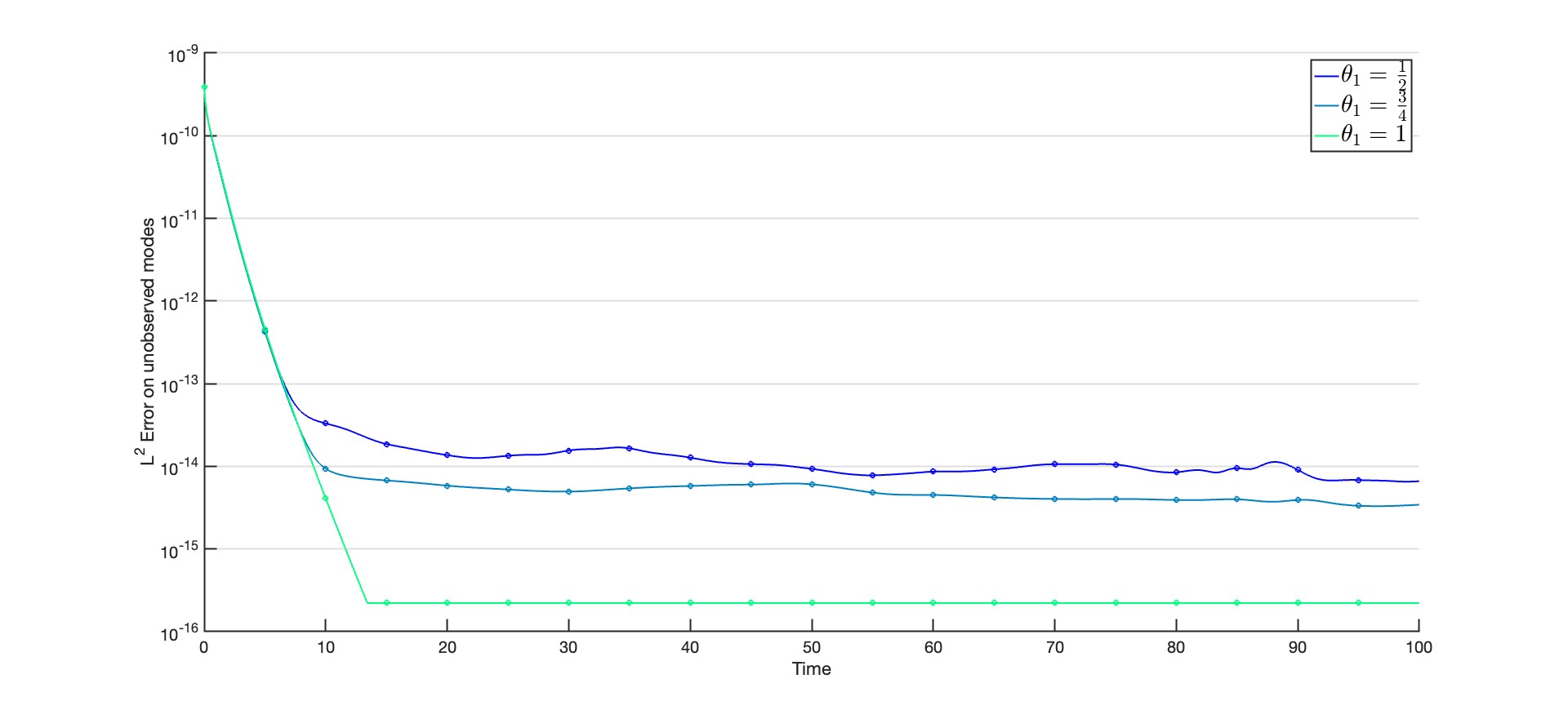}

    \end{subfigure}


    \caption{Error over time for symmetric direct-replacement intertwinements with $\theta_1$ as given in the legend. Here the initial condition for $v_2$ was given by $v_0^2 = v_0^1 + \varepsilon$, where $\varepsilon$ is given as in \cref{def:noise} with $\sigma^2 = 10^{-10}$.}

    \label{fig: symmetric-D-R long time}
\end{figure}


\subsection{Nudging Intertwinement}\label{sect:computations}

In this section, we describe an implementation of the  nudging intertwinement equations \eqref{eq:intertwined:nudge} and present the subsequent results. We focus only on the \textit{mutual nudging intertwinement}, i.e., $M=\begin{pmatrix}-\mu_1&\mu_1\\ \mu_2&-\mu_2\end{pmatrix}$ and the \textit{symmetric nudging intertwinement}, i.e., $M=\begin{pmatrix}-\mu_1&\mu_2\\\mu_2&-\mu_1\end{pmatrix}$, where $\mu_1,\mu_2\geq0$. For these intertwinements, we again implement the intertwined system according to \eqref{scheme:NSE}, but with the additional terms coming from the intertwining functions computed explicitly. We simulated these equations for various instances of the intertwining matrix, $M$, using spatial resolution, $N = 2^9$ and viscosity $\nu = 0.0005$. We again consider $g_1 = g_2 = f$, with $f$ being the time-independent forcing described in \cref{sect:numerical:tech}.

To initialize our equations we used a solution to 2D NSE that had been spun up from initial data $0$ up to time $10,000$. To generate the second initial profile we evolved this solution out an additional $100$ units of time at which point we found that the solutions appeared to be sufficiently decorrelated. This decorrelation can be observed in the error at the initial times in \cref{fig:mutual nudging error}, \cref{fig:mutual nudging error zoomed}, and \cref{fig:symmetric nudging error zoomed}.

\subsubsection{Mutual Nudging Intertwinement}\label{sect:computations:mutual}

For all of our simulations we fixed $\mu_1 = 50$ and varied $0\leq \mu_2\leq \mu_1$. We found that when utilizing nonnegative $\mu_2$ that all of the simulations behaved approximately the same way. In each case, we obtained exponential convergence of $v_1$ to $v_2$ at approximately the same rate. We observe the exponential decay in the error, split into the observed and unobserved modes, in \cref{fig:mutual nudging error}, where the plots are nearly indistinguishable. Upon zooming in on the initial development of the error, we see in \cref{fig:mutual nudging error zoomed} that the error converges exponentially in the initial period at rates which increase as $\mu_2$ increases in the initial period, before they transitions to a slower, but nevertheless exponential, decay rate. 

\begin{figure}

\begin{subfigure}[b]{.49\textwidth}
\centering

            \includegraphics[width=\textwidth]{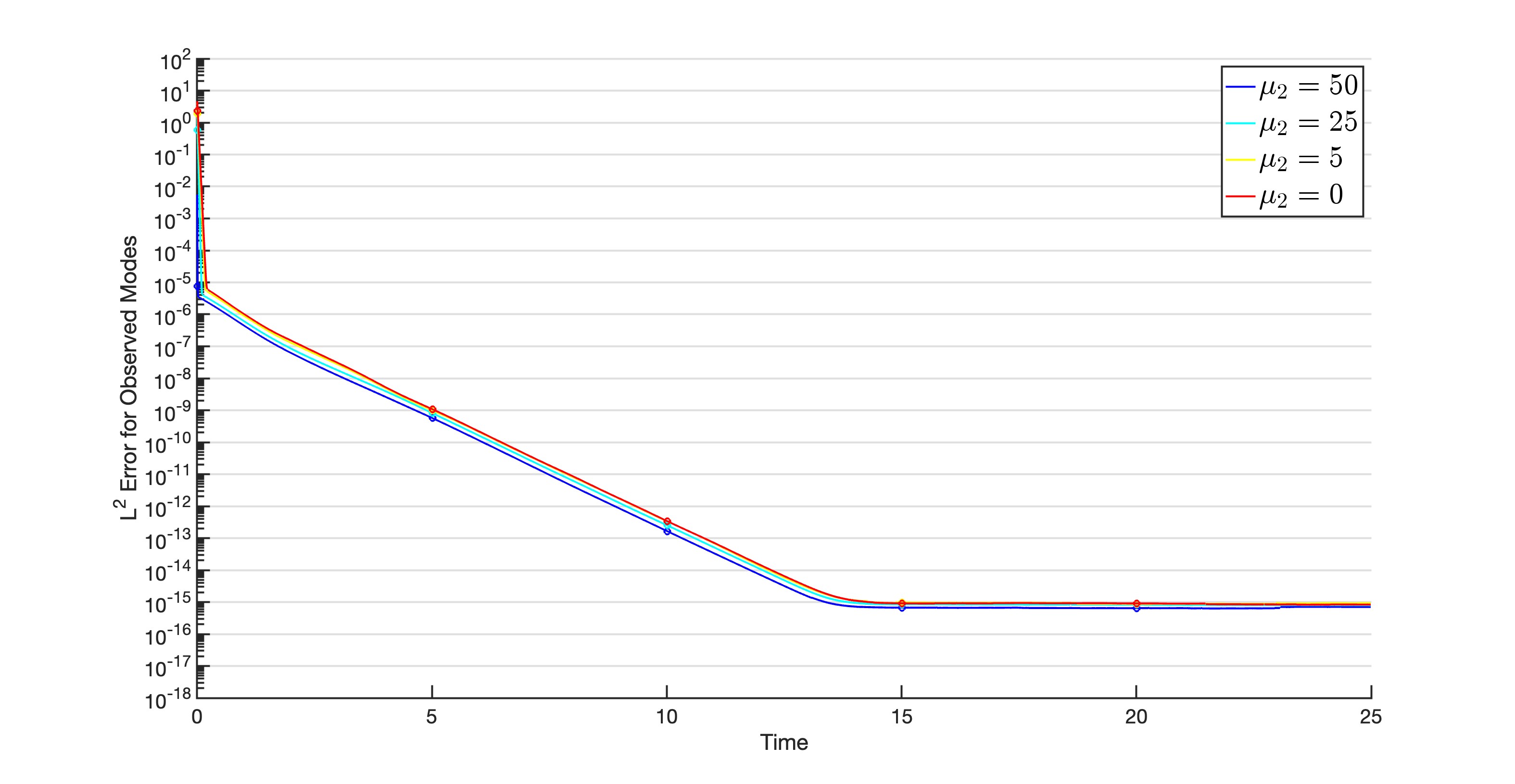}

    \end{subfigure}
    \begin{subfigure}[b]{.49\textwidth}
\centering

            \includegraphics[width=\textwidth]{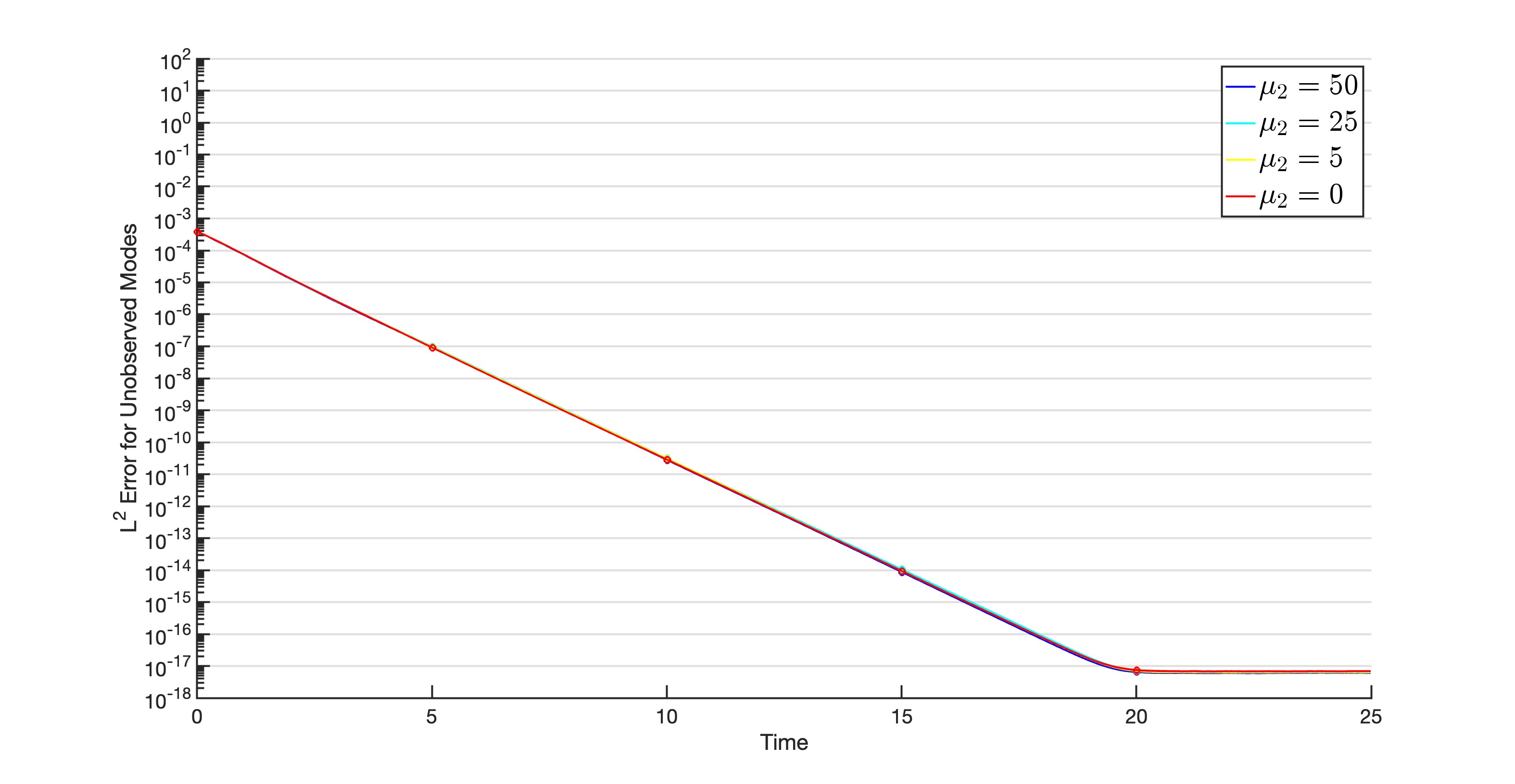}

    \end{subfigure}


    \caption{Error over time for mutual nudging intertwinements with $\mu_1 = 50$ and various non-negative values for $\mu_2$.}

    \label{fig:mutual nudging error}
\end{figure}

\begin{figure}

\begin{subfigure}[b]{.49\textwidth}
\centering

            \includegraphics[width=\textwidth]{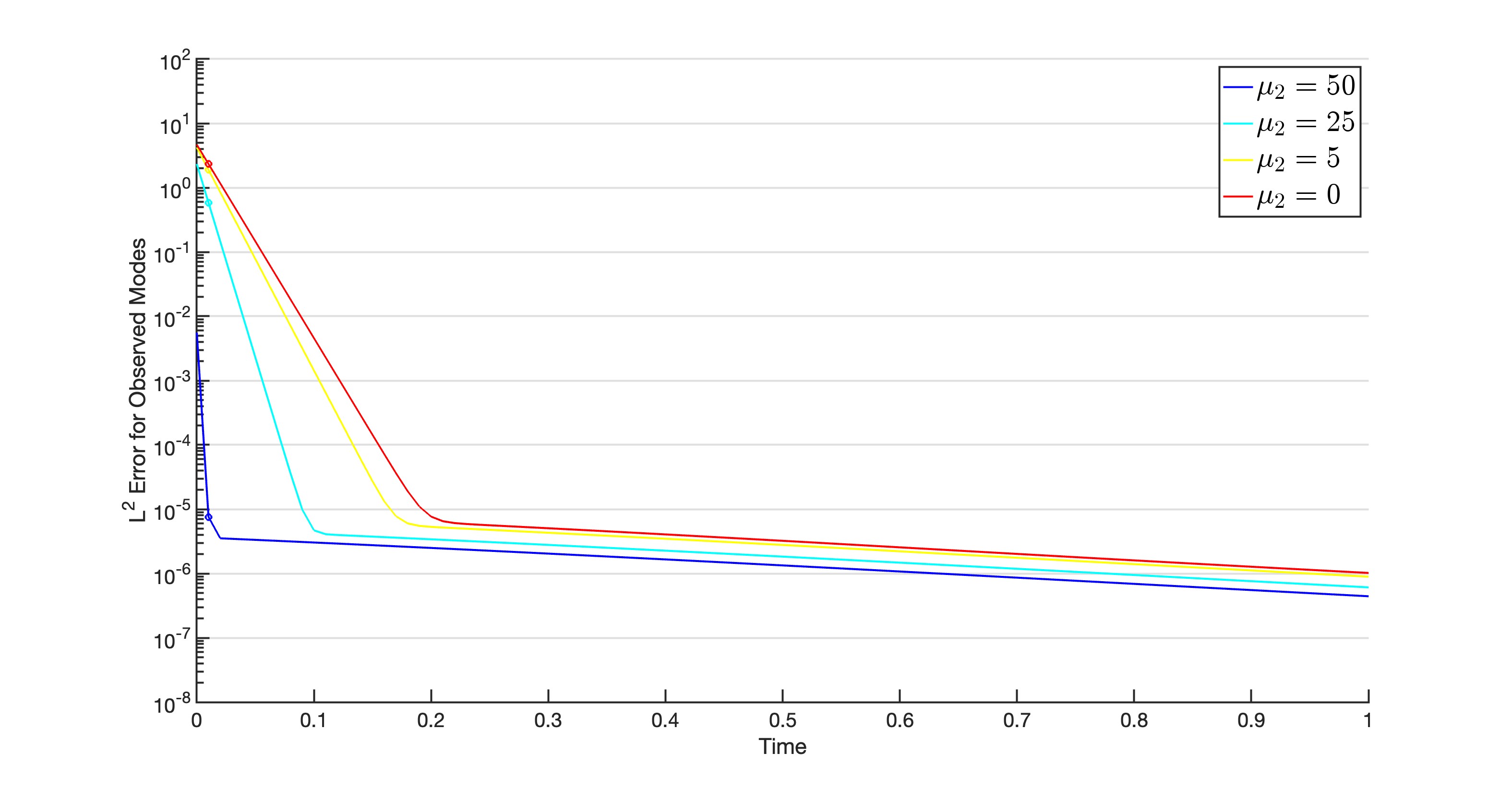}

    \end{subfigure}
    \begin{subfigure}[b]{.49\textwidth}
\centering

            \includegraphics[width=\textwidth]{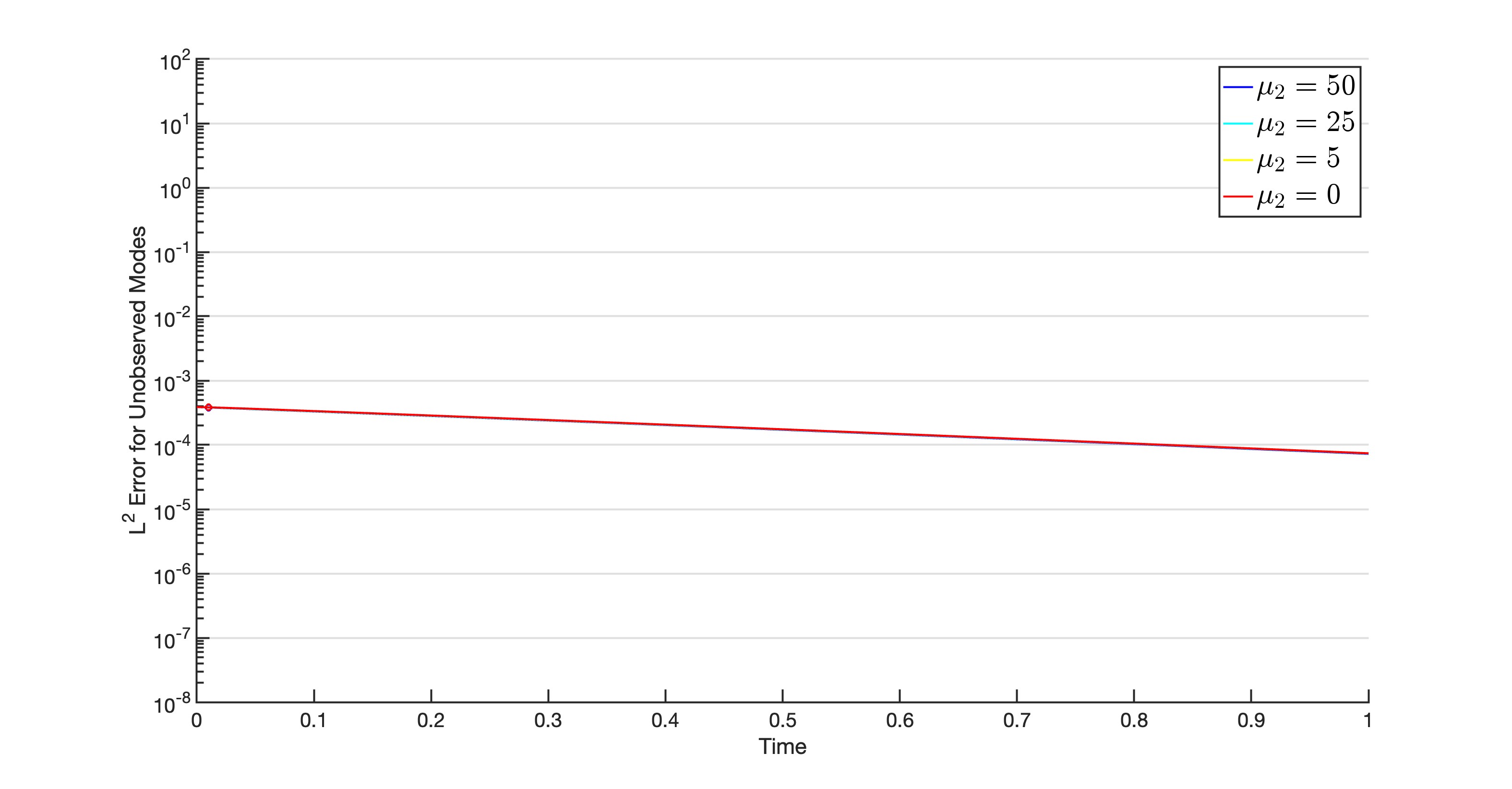}

    \end{subfigure}



\caption{Error over time for mutual nudging intertwinements with $\mu_1 = 50$ and various non-negative values for $\mu_2$. Zoomed in plot showing initial error development for \cref{fig:mutual nudging error}}
    \label{fig:mutual nudging error zoomed}
\end{figure}

\subsubsection{Symmetric Nudging Intetwinement}\label{sect:computations:symmetric}

For all of our simulations, we once again fixed $\mu_1 = 50$ and varied $0\leq \mu_2\leq \mu_1$. We found that all of the simulations behaved approximately the same, except for when $\mu_1 = \mu_2$, which is precisely when the smallest value of the eigenvalues of the intertwining matrix is zero. 
We see that in each case, except when $\mu_1 = \mu_2$, we obtained exponential convergence of $v_1$ to $v_2$ at approximately the same rate. In the case when $\mu_1 = \mu_2$ we nevertheless still obtained exponential synchronization between $v_1$ and $v_2$, but it occurred at a different rate than the other cases; this quality is found in both the ``observed" and ``unobserved" modes. The exponential decay in the error, split into the observed and unobserved modes, is presented in \cref{fig:symmetric nudging error}, where the plots are once again nearly indistinguishable. Upon zooming in on the early development of the error (see \cref{fig:symmetric nudging error zoomed}), we see that in each method, the cases where $\mu_2 \neq \mu_1$ initially exhibit a different rate of synchronization before quickly transitioning to a slower, but nevertheless exponential, decay rate.

\begin{figure}

\begin{subfigure}[b]{.49\textwidth}
\centering

            \includegraphics[width=\textwidth]{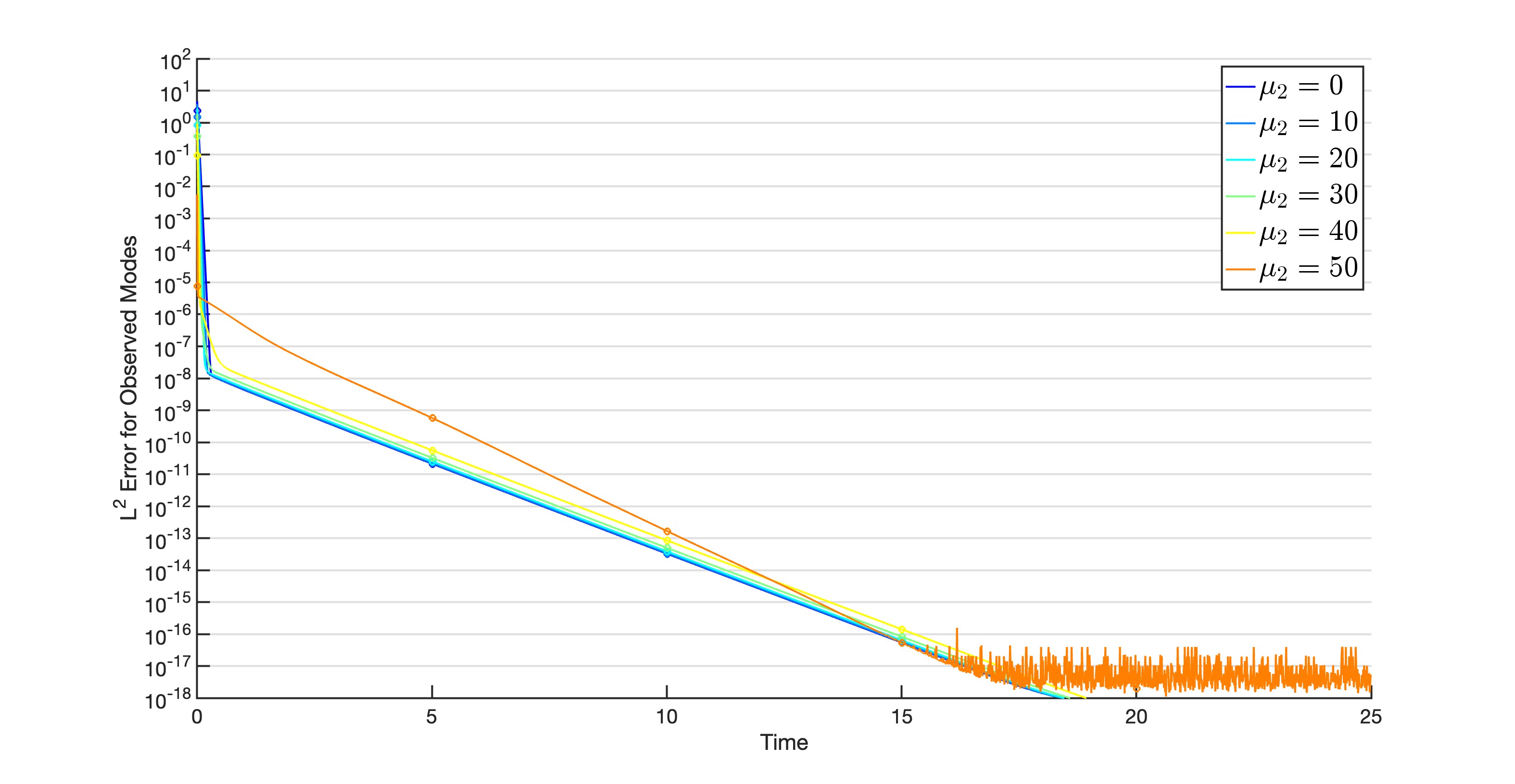}

    \end{subfigure}
    \begin{subfigure}[b]{.49\textwidth}
\centering

            \includegraphics[width=\textwidth]{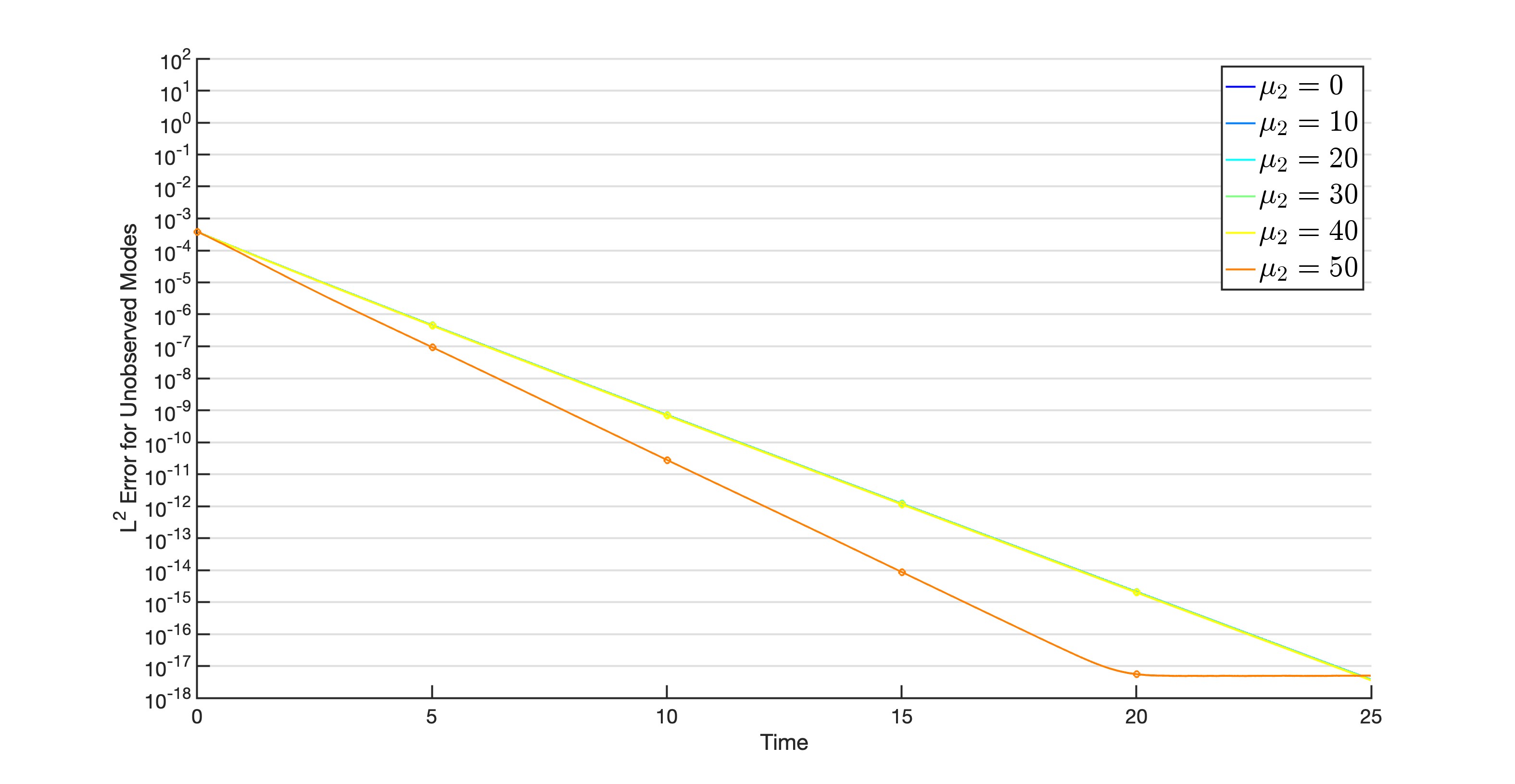}

    \end{subfigure}


    \caption{Error over time for symmetric nudging intertwinements with $\mu_1 = 50$ and various non-negative values for $\mu_2$.}

    \label{fig:symmetric nudging error}
\end{figure}

\begin{figure}

\begin{subfigure}[b]{.49\textwidth}
\centering

            \includegraphics[width=\textwidth]{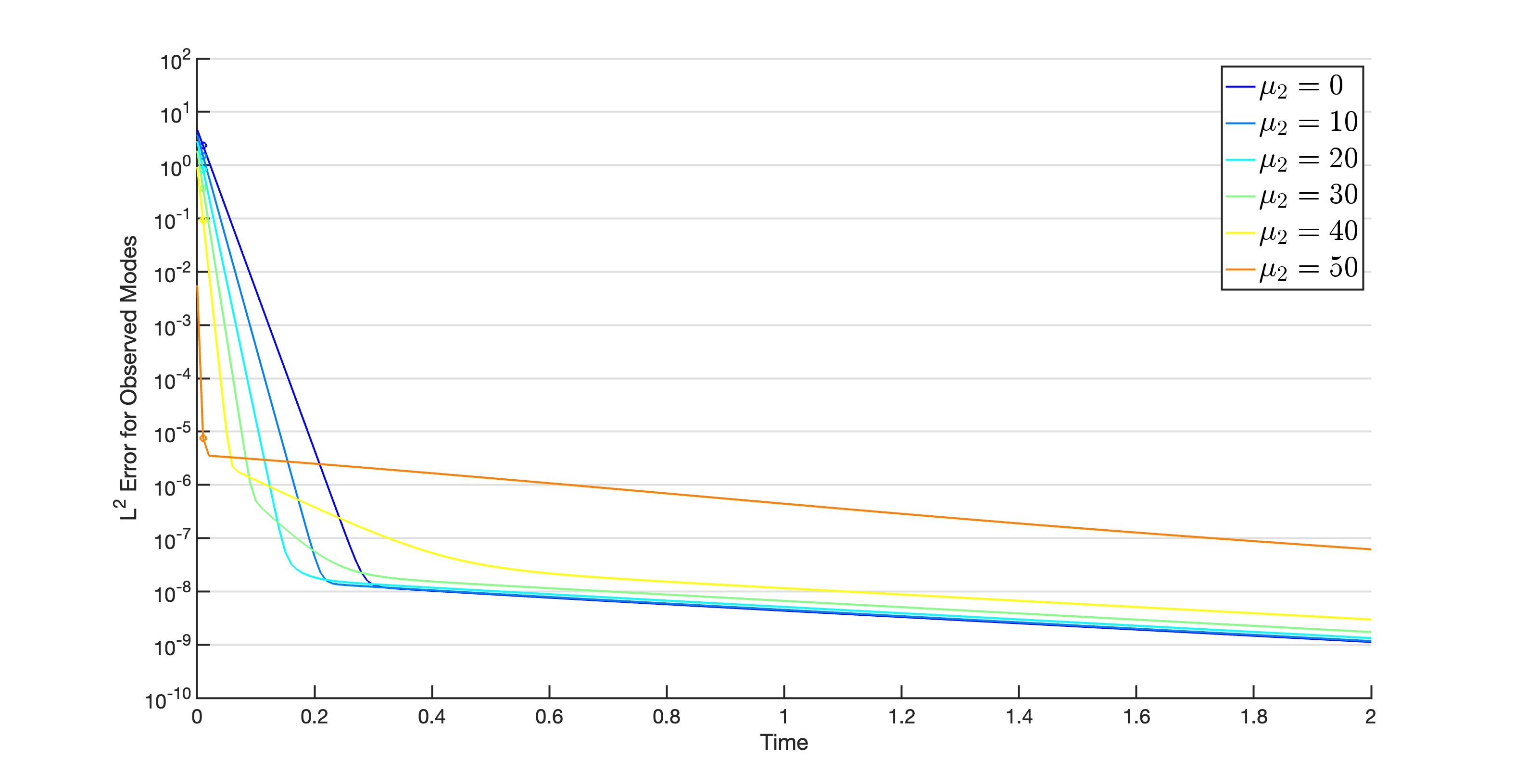}

    \end{subfigure}
    \begin{subfigure}[b]{.49\textwidth}
\centering

            \includegraphics[width=\textwidth]{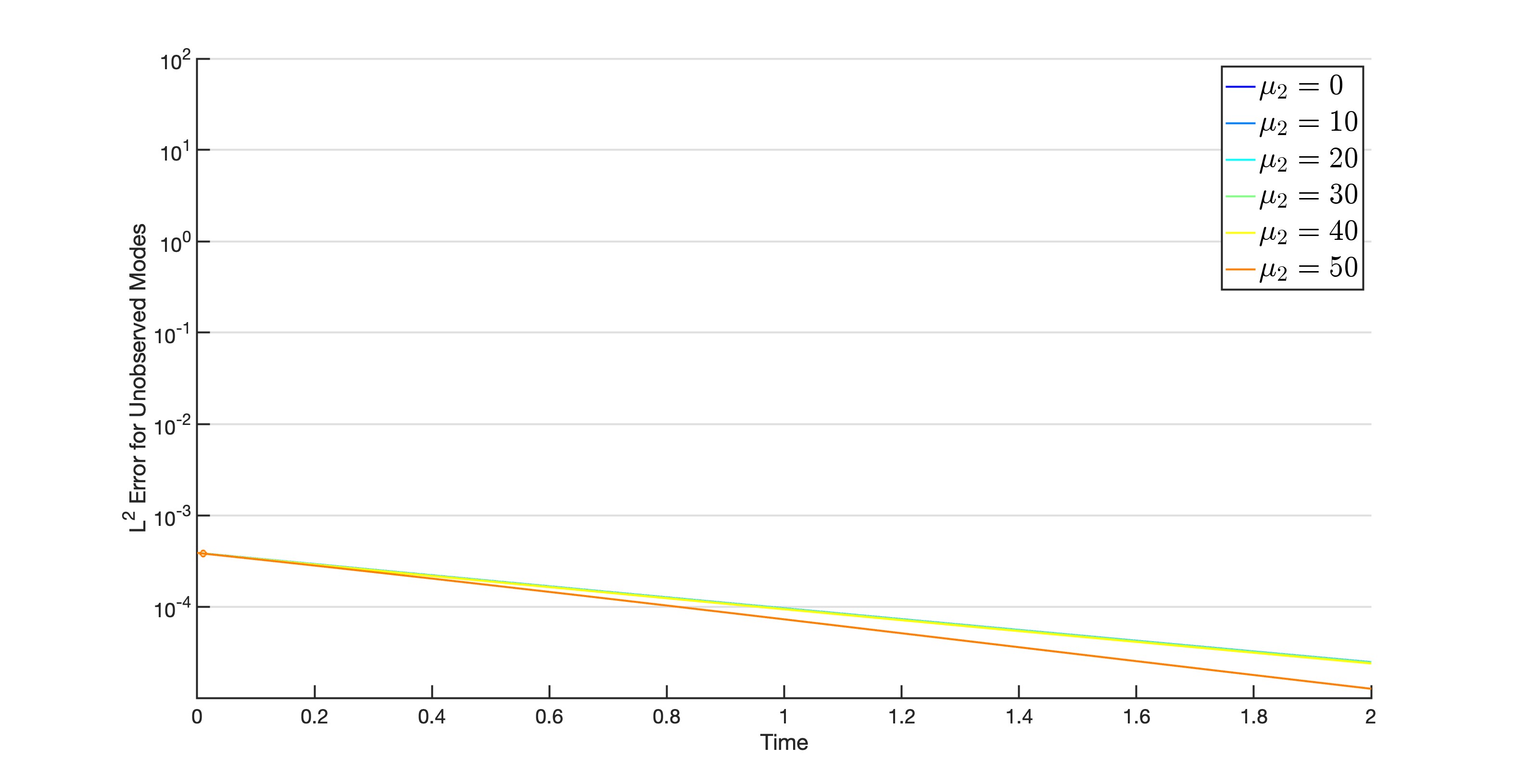}

    \end{subfigure}



\caption{Error over time for symmetric nudging intertwinements with $\mu_1 = 50$ and various non-negative values for $\mu_2$. Zoomed in plot showing initial error development for \cref{fig:symmetric nudging error}}
    \label{fig:symmetric nudging error zoomed}
\end{figure}

\subsection*{Ethics Declarations}

\subsubsection*{Competing interests} The authors declare no competing interests.

\subsubsection*{Data Availability} The datasets analyzed and the code used in their generation during the current study have been made available in the following publicly available repository: https://github.com/cvictor2/Data-Assimilation-Research/releases/tag/intertwinements.

\subsubsection*{Acknowledgments} The authors would like to thank Francesco Fosella for insightful discussions in the course of this work. E.C. was supported in part by the Department of Defense Vannevar Bush Faculty Fellowship, under
ONR award N00014-22-1-2790.  A.F. was supported in part by the National Science Foundation through DMS 2206493. V.R.M. was in part supported by the National Science Foundation through DMS 2213363, DMS 2206491, DMS 2511403, the Simons Foundation through MP-TSM 00014320, as well as the Dolciani Halloran Foundation.

\newcommand{\etalchar}[1]{$^{#1}$}
\providecommand{\bysame}{\leavevmode\hbox to3em{\hrulefill}\thinspace}
\providecommand{\MR}{\relax\ifhmode\unskip\space\fi MR }
\providecommand{\MRhref}[2]{%
  \href{http://www.ams.org/mathscinet-getitem?mr=#1}{#2}
}
\providecommand{\href}[2]{#2}

\vfill 

\noindent Elizabeth Carlson\\
{\footnotesize
Department of Mathematics\\
Oregon State University \\
Web: \url{https://sites.google.com/view/elizabethcarlsonmath}\\
Email: \url{carleliz@oregonstate.edu}\\
}

\noindent Aseel Farhat\\
{\footnotesize
Department of Mathematics \\
University of Virginia \\
Email: \url{af7py@virginia.edu}\\
}

\noindent Vincent R. Martinez\\
{\footnotesize
Department of Mathematics \& Statistics\\
CUNY Hunter College \\
Department of Mathematics \\
CUNY Graduate Center \\
Department of Computing \& Mathematical Sciences\\
California Institute of Technology\\
Web: \url{http://math.hunter.cuny.edu/vmartine/}\\
Email: \url{vrmartinez@hunter.cuny.edu}, \url{vrm@caltech.edu}\\
}

\noindent Collin Victor\\
{\footnotesize
Department of Mathematics \\
Texas A\&M University \\
Web: \url{https://collinvictor.me/}\\
Email: \url{collin.victor@tamu.edu}\\
}

\end{document}